\documentclass[a4paper,UKenglish,cleveref, autoref, thm-restate]{lipics-v2021}
\usepackage{amsmath, amsthm, amssymb, appendix, bm, graphicx, hyperref, mathrsfs}
\usepackage{multirow}
\usepackage{multicol}
\usepackage{makecell}
\usepackage[ruled]{algorithm2e}
\usepackage{caption}
\usepackage{subcaption}
\newcommand{\ii}[0]{\mathfrak{i}}

\title{Matchgate signatures under variable permutations} 

\author{Boning Meng\footnote{First author.}}{Key Laboratory of System Software (Chinese Academy of Sciences) and State Key Laboratory of Computer Science, Institute of Software, Chinese Academy of Sciences; University of Chinese Academy of Sciences, Beijing 100080, China}{mengbn@ios.ac.cn}{https://orcid.org/0009-0006-0088-1639}{}
\author{Yicheng Pan\footnote{Corresponding author.}}{Beihang University, Beijing 100191, China}{yichengp@buaa.edu.cn}{}{}
\authorrunning{Boning Meng and Yicheng Pan} 

\Copyright{Anonymous} 

\ccsdesc[100]{\textcolor{red}{Replace ccsdesc macro with valid one}} 

\keywords{Computational Complexity, Matchgate Signature, Counting CSP} 

\funding{Boning Meng is supported by National Key R\&D Program of China (2023YFA1009500), NSFC 61932002 and NSFC 62272448.}

\nolinenumbers 

\ArticleNo{}

\begin{document}

\maketitle

\begin{abstract}
In this article, we give a sufficient and necessary condition for determining whether a matchgate signature retains its property under a certain variable permutation, which can be checked in polynomial time. We also define the concept of permutable matchgate signatures, and use it to erase the gap between Pl-\#CSP and \#CSP on planar graphs in the previous study. We provide a detailed characterization of permutable matchgate signatures as well, by presenting their relation to symmetric matchgate signatures. In addition, we prove a dichotomy for Pl-$\#R_D$-CSP where $D\ge 3$ is an integer. 
\end{abstract}

\section{Introduction}

       Counting the number of perfect matchings in a graph (denoted as $\text{\#PM}$) is of great significance in counting complexity. 
        \#PM is motivated by the dimer problem in statistical physics \cite{kasteleyn1967graph,kasteleyn1961statistics,kasteleyn1963dimer,temperley1961dimer}, and two fundamental results emerge from this study. The first breakthrough occurred in 1961, when a polynomial time algorithm for $\text{\#PM}$ on planar graphs was developed by Kasteleyn, Temperley and Fisher \cite{kasteleyn1961statistics,temperley1961dimer}, known as the FKT algorithm. The second significant advancement occurred in 1979, when Valiant defined the complexity class $\text{\#P}$ and proved that $\text{\#PM}$ on general graphs is $\text{\#P}$-hard \cite{valiant1979complexity}. 
        It is the first natural counting problem discovered to be $\text{\#P}$-hard on general graphs and polynomial-time computable on planar graphs. Furthermore, in the complexity classification for counting constraint satisfaction problem (denoted as $\text{\#CSP}$) and Holant problem, $\text{\#PM}$ is also highly related to the cases in which these problems are \#P-hard in general but polynomial-time computable on planar graphs. In particular, the concept of matchgate signature, defined based on \#PM, is introduced to characterize such cases.
        
 
       A series of study have been settled to define and characterize matchgate signatures \cite{valiant2001quantum,valiant2002expressiveness,valiant2008holographic,cai2007someResultsMG,cai2009theoryMG,morton2010pfaffian,cai2013matchgates,margulies2016polynomial}. In particular, a so-called \textit{matchgate identity} (MGI) has been verified to be the necessary and sufficient condition for a signature to be a matchgate signature \cite{cai2007someResultsMG,cai2009theoryMG,cai2013matchgates}, which provides an algebraic way to characterize matchgate signatures. 

       In this article, we focus on what will happen if we permute the variables of a matchgate signature. To be more precious, we care about whether the resulting signature after the permutation remains a matchgate signature. There are two motivations for studying this issue. 
       
       Firstly, the study of counting complexity has been extended to a number of different graph classes. In particular, the complexity of \#PM on graphs that exclude a specific set of minors has been fully classified by \cite{curticapean2022parameterizing,thilikos2022killing}. On the other hand, in the study of signature grid theory, more attention are paid to the form of signatures. It is noticeable that complete complexity dichotomies have been proved for \#CSP and Pl-\#CSP, which means that each problem defined by a specific signature set $\mathcal{F}$, denoted by \#CSP$(\mathcal{F})$ or Pl-\#CSP$(\mathcal{F})$, is either polynomial-time computable or \#P-hard. Here, Pl-\#CSP asks that the incidence graph of the instance is planar, and in a certain planar embedding, the variables of each signature have to be positioned in a clockwise order. It is foreseeable that the study of these two topics will meet at some point\footnote{In fact, in an article to be appeared, we prove dichotomies for \#CSP on graphs that exclude a clique as a minor, greatly utilizing the results in this article.}.
       However, Pl-\#CSP differs from \#CSP on planar graphs, as the latter problem, denoted by \#CSP$(\mathcal{F})\langle\mathcal{PL}\rangle$, does not have the requirement of the order of the variables of each signature.  In this article, it will be demonstrated that our study erases this gap.
       
       Secondly, though matchgate signatures have already been characterized in an algebraic way (which is MGI), no alternative fully algebraic algorithm for the FKT algorithm has been identified. This is also an unresolved issue that was also noticed by \cite{carette2023compositionality} recently. The study of matchgate signatures under permutations has the potential to result in the creation of such an algorithm, as it has the potential to circumvent the utilization of the planar embedding.

        In this article, we always restrict ourselves to counting problems with Boolean domain and complex range, which means that each variable can only take value in $\{0,1\}$ and each signature has a range over $\mathbb{C}$. We achieve the following results in this article. 
        \begin{enumerate}
            \item We prove a dichotomy for Pl-\#$R_D$-CSP (Theorem \ref{plRDCSP}), which is Pl-\#CSP where each variable can only appear at most $D\ge 3$ times.
            \item We present a sufficient and necessary condition for determining whether a matchgate signature retains its property under a specific variable permutation (Theorem \ref{lem:4andpiMG}). 
            \item We define the concept of permutable matchgate signatures (Definition \ref{def:MP}), denoted as the set $\mathscr{M}_P$, and use this concept to erase the gap between \#CSP$\langle\mathcal{PL}\rangle$ and Pl-\#CSP (Theorem \ref{thmCSPpl}).
            
            \item We characterize the permutable matchgate signature in detail by showing the connection between it and the symmetric matchgate signature.
            \begin{enumerate}
                \item  We demonstrate that each permutable matchgate signature can be realized by symmetric matchgate signatures in a star-like way (Theorem \ref{lem:SymGadRealize}).
                \item For each permutable matchgate signature that may lead to \#P-hardness in the setting of \#CSP, we demonstrate that it can realize a symmetric matchgate signature that may also lead to \#P-hardness (Theorem \ref{lem:AsymtoSym}).
            \end{enumerate}
         \end{enumerate}
  
Result 1 is obtained by some slight modifications on the proof in \cite{cai2014complexity}, ensuring that the reductions persist for planar cases. It is remarkable that in the proof of Result 2, we are inspired by an alternative proof of the theorem that matchgate signatures possess MGI, given by Jerrum and recorded in \cite[Section 4.3.1]{Cai_Chen_2017}. By dividing the summation in MGI into appropriate parts, we prove this result by a non-trivial induction. In Result 3 and 4, we define and characterize permutable matchgate signatures. It is noteworthy that there are actually three types of signatures (Pinning, Parity, Matching) possess different properties, but they are all related to the corresponding symmetric matchgate signatures by our proof.

    This paper is organized as follows. 
In Section \ref{preliminareis}, we introduce the preliminaries needed in our proof. 
In Section \ref{sec:preprocessing}, we present the dichotomies for \#CSP. 
In Section \ref{secPMG}, we characterize matchgate signatures under permutations. 
In Section \ref{conclusion}, we conclude our paper.

\section{Preliminaries}\label{preliminareis}

\subsection{Counting problems}\label{preCSP}
      For a string $\alpha=\alpha_1\dots\alpha_k \in \{0,1\}^k$, the \textit{Hamming weight} of $\alpha$ is the number of $1$s in $\alpha$, denoted as $HW(\alpha)$. We use $\overline{\alpha}$ to denote the string that differs from $\alpha$ at every bit, which means $\alpha_i+\overline{\alpha}_i=1$ for each $1\le i\le k$. A \textit{signature}, or a \textit{constraint function}, is a function $f:\{0,1\}^k \to\mathbb{C}$ that maps a string of length $k$ to a complex number, where $k$ is denoted as the \textit{arity} of $f$.   A signature $f$ is said to be \textit{symmetric} if the value of $f$ depends only on the Hamming weight of the input. A symmetric signature $f$ of arity $k$ can be denoted as $[f_0,f_1,...,f_k]_k$, or simply $[f_0,f_1,...,f_k]$ when $k$ is clear from the context, where for $0\le i\le k$, $f_i$ is the value of $f$ when the Hamming weight of the input is $i$. For $c\in \mathbb{C}$, we also use the notation $c[f_0,f_1,...,f_k]$ to denote the signature $[cf_0,cf_1,...,cf_k]$. We use $\leq_T$ and $\equiv_T$ to respectively denote polynomial-time Turing reduction and equivalence. 
     
     We denote by $f^{x_i=c}$ the signature that pins the $i$th variable to $c\in\{0,1\}$:
     
     $$f^{x_i=c}(x_1,...,x_{i-1},x_{i+1},...,x_{k})=f(x_1,...,x_{i-1},c,x_{i+1},...,x_{k})$$
     For a string $\alpha\in \{0,1\}^q, q\le k$, we also define $f^\alpha=f^{x_1=\alpha_1,x_2=\alpha_2,...,x_q=\alpha_q}$, where $\alpha_i$ is the $i$th bit of $\alpha$ for $1\le i\le q$. That is, $f^\alpha$ is obtained from $f$ by pinning the first $q$ bits of $f$ to $\alpha_1,...,\alpha_q$. Equivalently speaking, $\alpha$ can be seen as an assignment to the $q$ variables $\alpha:\{1,\dots,q\}\to \{0,1\}$. In the case of a slight overuse, the notation $f^\alpha$ is sometimes employed to indicate the signature by pinning specific $q$ bits of $f$ to $\alpha_1,...,\alpha_q$ when the $q$ bits are clear from the context. 
     
    Two frameworks of counting problems, namely $\text{\#CSP}$ problem and $\text{Holant}$ problem, are of great significance as they are capable of expressing a wide range of problems. For example, problems in \cite{vertigan2005computational,cai2017holographic,cai2012spin,cai2012gadgets} can be expressed in the form of a $\text{\#CSP}$ problem and characterized by the result in \cite{guo2020complexity}. Furthermore, $\text{\#CSP}$ problem and $\text{\#PM}$ problem can be expressed in the form of $\text{Holant}$ problem. We present their definitions as follows.
    \subsubsection{\#PM}\label{PMstudy}
        For a graph $G=(V,E)$, a \textit{matching} is an edge set $M\subseteq E$ such that no pair of edges in $M$ shares a common endpoint.  Besides, if the vertices that $M$ contains are exactly $V(M)=V$, then $M$ is denoted as a \textit{perfect matching} of $G$.
        
        An instance of \#PM is a graph $G=(V,E)$ with weighted edges $w:E\to \mathbb{C}$. The weight of a matching $M$ is $w(M)=\prod_{e\in M}w(e)$. The output of the instance is the sum of the weights of all perfect matchings in $G$:
        
        $$\text{\#PM}(G)=\sum_{M: M\text{ is a perfect matching of }G}w(M)$$
        When $w(e)=1$ for each $e\in E$, the output of the instance is just the number of perfect matchings in the graph, and we denote this kind of problems as standard \#PM.


     \subsubsection{\#CSP problems}
       A \textit{counting constraint satisfaction problem} $\text{\#CSP}(\mathcal{F})$ \cite{creignou2001complexity} requires the value of an instance, which is the sum of the values over all configurations. Here, $\mathcal{F}$ is a fixed and finite set of signatures. An instance of $\text{\#CSP}(\mathcal{F})$ is specified as follows:
      
      \begin{definition}
            An \textit{instance $I$ of $\text{\#CSP}(\mathcal{F})$} has $n$ variables and $m$ signatures from $\mathcal{F}$ depending on these variables. The value of the instance then can be written as
            
        $$Z(I)=\sum_{(x_1,...,x_n)\in \{0,1\}^n}\prod_{1\le i\le m} f_i(x_{i_1},...,x_{i_k})$$
        
        where $f_1,\dots,f_m$ are signatures in $I$ and $f_i$ depends on $x_{i_1},...,x_{i_k}$ for each $1\le i\le m$.
          \label{defCSP}
      \end{definition}
        The \textit{underlying graph} of a $\text{\#CSP}(\mathcal{F})$ instance $I$ is a bipartite graph $G=(U,V,E)$, where for every constraint $f$ there is a $u_f\in U$, for every variable $x$ there is a $v_x\in V$, and $(u_f,v_x)\in E$ if and only if $f$ depends on $x$. Sometimes we also denote the value $Z(I)$ as $Z(G)$ for convenience. If each constraint function in $\mathcal{F}$ is restricted to be symmetric, we denote this kind of problem as symmetric \#CSP, or sym-\#CSP for short. 
        
        If we restrict the maximum degree of the vertices in $V$ to be no more than a constant $D$, this kind of problem is denoted as $\#R_D$-CSP in \cite{cai2014complexity}.
        If we restrict $G$ to be a planar graph with a certain planar embedding, in which the variables that $f_u$ depends are ordered counterclockwise for each $u\in U$, then we denote this problem as Pl-\#CSP.
        Pl-$\#R_D$-CSP is defined similarly.
        \begin{remark}  \label{remembd}
       It is worth noticing that $\text{Pl-\#CSP}(\mathcal{F})$ is not exactly $\text{\#CSP}(\mathcal{F})\langle \mathcal{PL}\rangle$ (\#CSP on planar graphs). In $\text{Pl-\#CSP}(\mathcal{F})$, it is not sufficient for the underlying graph of each instance to be planar; rather, each signature within the instance must also be positioned in an appropriate manner. %
   \end{remark}
        \subsubsection{Holant problems}
        A \textit{Holant problem} $\text{Holant}(\mathcal{F})$ can be seen as a $\text{\#CSP}(\mathcal{F})$ problem with the restriction that all the variables must appear exactly twice. 
        
        \begin{definition}
        An instance of $\text{Holant}(\mathcal{F})$ has an underlying graph $G=(V,E)$. Each vertex $v\in V$ is assigned a signature from $\mathcal{F}$ and each edge in $E$ represents a variable.  Here, $\mathcal{F}$ is a fixed set of signatures and usually finite. The signature assigned to the vertex $v$ is denoted as $f_v$. An assignment of $E$ is a mapping $\sigma:E\to \{0,1\}$, which can also be expressed as an assignment string $\sigma\in \{0,1\}^{|E|}$, and the value of the assignment is defined as
        
         $$\omega(\sigma)=\prod_{v\in V} f_v(\sigma)$$
        where $f_v(\sigma)=f_v(\sigma({e_{v_1}}),...,\sigma({e_{v_k}}))$ and $v$ is incident to $e_{v_1},...,e_{v_k}$.
        \par The output of the instance, or the value of $G$, is the sum of the values of all possible assignments of $E$, denoted as:
        
        $$Z(G)=\sum_{\sigma\in \{0,1\}^{|E|}}\omega(\sigma)$$ 
            \label{defHol}
        \end{definition}
        
        Furthermore, we use $\text{Holant}(\mathcal{F}_1|\mathcal{F}_2)$ represents $\text{Holant}(\mathcal{F}_1\cup \mathcal{F}_2)$ with the restriction that the underlying graph $G=(U,V,E)$ is bipartite, and each vertex $u\in U$ is assigned a signature from $\mathcal{F}_1$ while each vertex $v\in V$ is assigned a signature from $\mathcal{F}_2$.
         We denote by $\mathcal{EQ}$ the set of all equality functions. In other words, $\mathcal{EQ}=\{=_k|k\ge1\}$ where $=_k$ is the signature $[1,0,...,0,1]_k$. We also denote $\{=_k|1\le k\le D\}$ by $\mathcal{EQ}_{\le D}$ for each integer $D\ge 1$.  
         By definition, we have the following lemma. Also see Figure \ref{fig:CSPholant} for an example.
         \begin{lemma}
         Let $\mathcal{C}$ be an arbitrary graph class, $\mathcal{F}$ be an arbitrary signature set and $D\ge 1$ be an integer. Then,
         
             $$\text{\#CSP}(\mathcal{F})\langle \mathcal{C}\rangle\equiv_T\text{Holant}(\mathcal{F}|\mathcal{EQ})\langle \mathcal{C}\rangle$$ 
              $$\#R_D\text{-CSP}(\mathcal{F})\langle \mathcal{C}\rangle\equiv_T\text{Holant}(\mathcal{F}|\mathcal{EQ}_{\le D})\langle \mathcal{C}\rangle$$ 
             \label{lemcsp=hol}
         \end{lemma}

         Besides, standard \#PM is exactly $\text{Holant}(\{[0,1,0,...,0]_k|k\in \mathbb{N}^+\})$. Suppose that $G=(V,E)$ is a graph with each vertex of degree $k$ assigned $[0,1,0,...,0]_k$. For each  assignment $\sigma$ of $E$, $\omega(\sigma)$ can only be either 0 or 1 since the value of each signature can only provide a multiplier with value 0 or 1. Furthermore, $\omega(\sigma)=1$ if and only if all edges assigned the value 1 form a perfect matching of $G$. Consequently, we have \#PM$(G)=Z(G)$. In addition, by exchanging 0 and 1, standard \#PM can also be expressed as $\text{Holant}(\{[0,...,0,1,0]_k|k\in \mathbb{N}^+\})$. Similarly, \#PM is exactly $\text{Holant}(\{[0,...,0,1,0]_k|k\in \mathbb{N}^+\}|\{[1,0,c]|c\in \mathbb{C}\})$ since $\{[1,0,c]|c\in \mathbb{C}\}$ is capable of expressing the weight of each edge.
     \begin{figure}
            \centering
            \includegraphics[height=0.2\textheight]{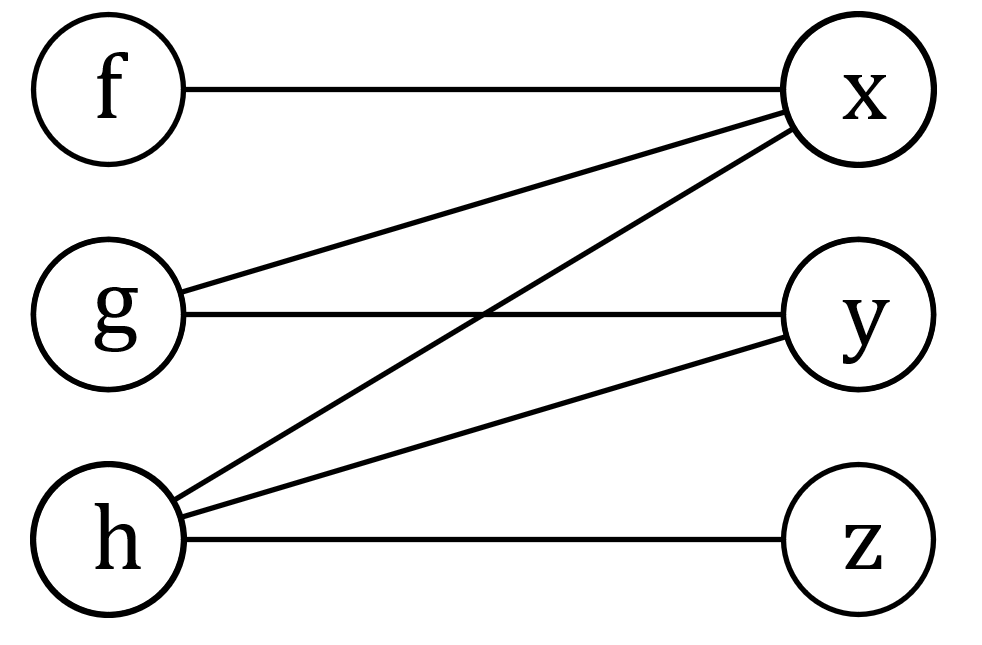}
            \caption{The underlying graph of a $\text{\#CSP}(\mathcal{F})$ instance if $f,g,h\in \mathcal{F}$ and $x,y,z$ are 3 variables; or that of a $\text{Holant}(\mathcal{F}|\mathcal{EQ})$ instance if we treat $x,y,z$ as $=_3,=_2,=_1$ and the edges as the variables instead.}
            \label{fig:CSPholant}
        \end{figure}

\subsection{Reduction methods}

\subsubsection{Constructing gadgets}

  A \textit{gadget} of $\text{Holant}(\mathcal{F})$ has an underlying graph $GG=(V,E,D)$, where $E$ is the set of normal edges and $D$ is the set of edges with only one endpoint, called \textit{dangling edges} \footnote{In order to differentiate from the notation of a graph, we use two capital letters to represent a gadget.}. Each vertex in $GG$ is still assigned a signature from $\mathcal{F}$. A signature $f$ of arity $|D|$ is said to be \textit{realized} by $GG=(V,E,D)$, if for each assignment $\alpha : D \to \{0,1\}$, $f(\alpha)=\sum_{\sigma\in \{0,1\}^{|E|}}\omega(\alpha\sigma)$, where $\alpha\sigma$ is the assignment of edges in $D\cup E$. In this case, we also say $f$ can be realized by $\mathcal{F}$. 
        By constructing gadgets with existing signatures, we are able to realize desired signatures.

        \begin{lemma}
            If $f$ can be realized by $\mathcal{F}$, then $\text{Holant}(\mathcal{F})\equiv_T\text{Holant}(\mathcal{F}\cup\{f\})$.
        \end{lemma}
        
        Also, we present some derivative concepts related to the concept of a gadget. A \textit{left-side gadget} of $\text{Holant}(\mathcal{F}_1|\mathcal{F}_2)$ has a bipartite underlying graph $GG=(U,V,E,D)$, where each vertex $u\in U$ is assigned a signature from $\mathcal{F}_1$, each vertex $v\in V$ is assigned a signature from $\mathcal{F}_2$, $E$ is the set of normal edges and $D$ is the set of dangling edges. Furthermore, the endpoint of each dangling edge must belong to $U$. It is easy to verify that, if $f$ can be realized by $GG$, then $\text{Holant}(\mathcal{F}_1|\mathcal{F}_2)\equiv_T\text{Holant}(\mathcal{F}_1\cup\{f\}|\mathcal{F}_2)$. The \textit{right-side gadget} is defined similarly except that the endpoint of each dangling edge must belong to $V$.
        

\subsubsection{Holographic Transformation}
     Let $T$ be a binary signature, and we denote the two dangling edges corresponding to the input variables of it as a left edge and a right edge. Its value then can be written as a matrix $T=\begin{pmatrix}
          t_{00}&t_{01}\\
          t_{10}&t_{11}
      \end{pmatrix}$, where $t_{ij}$ is the value of $T$ when the value of left edge is $i$ and that of the right edge is $j$.

    This notation is conducive to the efficient calculation of the gadget's value. Let us consider two binary signatures, $T$ and $P$, with the right edge of $T$ connected to the left edge of $P$. $T$ and $P$ now form a binary gadget. Subsequently, it can be demonstrated that the value of the resulting gadget is precisely $TP$, which represents the matrix multiplication of $T$ and $P$.
      
      For a signature $f$ of arity $n$ and a binary signature $T$, we use $Tf/fT$ to denote the signature ``$f$ transformed by $T$'', which is a signature of arity $n$ obtained by connecting the right/left edge of $T$ to every dangling edge of $f$. For a set $\mathcal{F}$ of signatures, we also define $T\mathcal{F}=\{Tf|f\in \mathcal{F}\}$. Similarly we define $\mathcal{F}T$. The following theorem demonstrates the relationship between the initial and transformed problems: 
      \begin{theorem}[{Holographic Transformation}{\cite{valiant2008holographic,cai2007valiant}}]
          $Holant(\mathcal{F}|\mathcal{G}) \equiv_T Holant(\mathcal{F}T^{-1}|T\mathcal{G})$
          \label{thmHT}
      \end{theorem}
      Let $H_2=\begin{pmatrix}
          1&1\\
          1&-1
      \end{pmatrix}$. For a set of signatures $\mathcal{F}$, we use $\widehat{\mathcal{F}}$ to denote $H_2\mathcal{F}$. As $H_2^{-1}=\frac{1}{2}H_2$, by Theorem \ref{thmHT} we have:
      
      $$Holant(\mathcal{F}|\mathcal{G}) \equiv_T Holant(\widehat{\mathcal{F}}|\widehat{\mathcal{G}})$$

     We additionally present the following fact as a lemma for future reference.
     \begin{lemma}
         For each $k\ge 1$, $\widehat{=_k}=[1,0,1,0,1,0,\dots]_k$. For example, $\widehat{=_1}=[1,0], \widehat{=_2}=[1,0,1], \widehat{=_3}=[1,0,1,0]$. Consequently, $\widehat{\mathcal{EQ}}= \{[1,0,1,0,1,0,\dots]_k| k\ge 1\}$ and for an integer $D\ge 1$, $\widehat{\mathcal{EQ}_{\le D}}= \{[1,0,1,0,1,0,\dots]_k| 1\le k\le D\}$.
         \label{transEQ}
     \end{lemma}

\subsection{Matchgate signatures}
 \begin{definition}
          A matchgate is a planar graph $G=(V,E)$ with weighted edges $w:E\to \mathbb{C}$, and together with some external nodes $U\subseteq V$ on its outer face labelled by $\{1,2,...,|U|\}$ in a clockwise order. The signature $f$ of a matchgate $G$ is a Boolean signature of arity $|U|$ and for each $\alpha\in \{0,1\}^{|U|}$,
          
          $$f(\alpha)=\#PM(G-X),$$
          where $X\subseteq U$ and a vertex in $U$ with label $i$ belongs to $X$ if and only if the $i$th bit of $\alpha$ is 1.
          \par A signature $f$ is a matchgate signature if it is the signature of some matchgate.
           $\mathscr{M}$ denotes all the matchgate signatures.
          \label{defMG}
      \end{definition}
      
      Some particularly useful properties of matchgates are as follows: 
      \begin{lemma}[\cite{cai2013matchgates}]
          A matchgate signature of arity $k$ can be realized by a matchgate with at most $O(k^4)$ vertices, which can be constructed in $O(k^4)$ time.
          \label{lemMG}
      \end{lemma}
          Furthermore, an algebraic representation of the matchgate signatures can be formulated.
\begin{theorem}[MGI]
    Suppose $f$ is a signature of arity $k$. Then $f$ is a matchgate signature if and only if the following identity, denoted as MGI, is satisfied:

    For each $\beta,\gamma\in\{0,1\}^k$, let $P=\{1\le j\le k| \beta_j\neq\gamma_j\}$, $l=|P|$. Let $p_j\in P$ be the $j$th smallest number in $P$ and let $e_{p_j}\in\{0,1\}^k$ denotes a string with a 1 in the $p_j$th index, and 0 elsewhere. Then
    
    $$\sum_{j=1}^n(-1)^jf(\beta\oplus e_{p_j})f(\gamma\oplus e_{p_j})=0$$
\end{theorem}

\section{Dichotomies for \#CSP}\label{sec:preprocessing}

\subsection{Previous results for \#CSP}
We introduce two more sets of signatures related to the dichotomies for \#CSP.
\begin{definition}
        A signature $f$ with arity $k$ is of \textit{affine type} if it has the form:
        $$\lambda\chi_{AX=b}\cdot \mathfrak{i}^{\sum_{1\le i\le n}a_i{x_i}^2+2\sum_{1\le i< j\le n} b_{ij}x_ix_j}$$
        where $\mathfrak{i}$ denotes the imaginary unit, $a_i,b_{ij}\in\{0,1\},\lambda\in \mathbb{C}$, $AX=b$ is a system of linear equations on $\mathbb{Z}_2$ and
        \begin{equation}
            \chi_{AX=b}=\begin{cases}
                1, & \text{if } AX=b; \\
                0, & \text{otherwise}.\notag
            \end{cases}
        \end{equation}
        \par $\mathscr{A}$ denotes the set of all the signatures of affine type.
        \label{defA}
    \end{definition}
       \begin{definition}
        A signature $f$ is an $\mathcal{E}$\textit{-signature} if it has value 0 except for 2 complementary supports. A signature $f$ with arity $k$ is of \textit{product type} if it can be expressed as a tensor product of $\mathcal{E}$-signatures.
       
        \par $\mathscr{P}$ denotes the set of all the signatures of product type.
         \label{defP}
    \end{definition}

      The existing dichotomy for \#CSP, $\#R_D$-CSP and Pl-\#CSP can be stated as follows:  
\begin{theorem}[\cite{cai2014complexity}]
        If $\mathcal{F}\subseteq \mathscr{A}$ or $\mathcal{F}\subseteq \mathscr{P}$, $\text{\#CSP}(\mathcal{F})$ is computable in polynomial time; otherwise it is $\text{\#P}$-hard.
        \label{genCSP}
    \end{theorem}

\begin{theorem}[\cite{cai2014complexity}]
        Suppose $D\ge 3$ is an integer. If $\mathcal{F}\subseteq \mathscr{A}$ or $\mathcal{F}\subseteq \mathscr{P}$, $\#R_D\text{-CSP}(\mathcal{F})$ is computable in polynomial time; otherwise it is $\text{\#P}$-hard.
        \label{genRDCSP}
    \end{theorem}

    \begin{theorem}[\cite{cai2017holographicuni}]
        If $\mathcal{F}\subseteq \mathscr{A}$ or $\mathcal{F}\subseteq \mathscr{P}$ or $\mathcal{F}\subseteq \widehat{\mathscr{M}}$, $\text{Pl-\#CSP}(\mathcal{F})$ is computable in polynomial time; otherwise it is $\text{\#P}$-hard.
        \label{plCSP}
    \end{theorem}

\subsection{Dichotomy for Pl-\#$R_D$-CSP}\label{sec:easyplrdcsp}

The dichotomy for Pl-$\#R_D$-CSP is stated as follows.

\begin{theorem} \label{plRDCSP}
    Suppose $D\ge3$ is an integer. If $\mathcal{F}\subseteq \mathscr{A}$ or $\mathcal{F}\subseteq \mathscr{P}$ or $\mathcal{F}\subseteq \widehat{\mathscr{M}}$, Pl-$\#R_D$-CSP$(\mathcal{F})$ is computable in polynomial time; otherwise it is $\text{\#P}$-hard.
\end{theorem}

We delay the full proof of Theorem \ref{plRDCSP} to Appendix \ref{sec:plrdcsp}, as it is similar to that of Theorem \ref{genRDCSP} originally presented in in \cite[Section 6]{cai2014complexity}. Here, we only remark the major differences between the two proofs.
\begin{itemize}
    \item All the gadgets in our proof are constructed in the setting of Pl-$\#R_D$-CSP instead of $\#R_D$-CSP;
    \item Unlike \cite[Lemma 6.5]{cai2014complexity}, we realize $[0,0,1]=[0,1]^{\otimes 2}$ instead of $[0,1]^{\otimes m}$ in Lemma  \ref{lem:plr3obtaind1};
    \item Unlike \cite[Lemma 6.2]{cai2014complexity}, instead of realizing a single non-degenerate binary signature $h$, we realize $h^{\otimes 2}$ and use them to form a non-degenerate binary signature $g=h^Th$ in Lemma \ref{lem:plr3obtainh}.
\end{itemize}

\section{Matchgate signatures under variable permutations}\label{secPMG}

In Section \ref{sec:MGnormal}, we introduce the concept of the normalized matchgate signature, which simplifies the form of matchgate signatures. 
In Section \ref{sec:PermCheck}, we prove that in polynomial time we may check whether a matchgate signature under a given permutation remains a matchgate signature. 
In Section \ref{sec:easyMP}, we introduce the concept of permutable matchgate signatures, and show that how this concept erase the gap between Pl-\#CSP and \#CSP$\langle\mathcal{PL}\rangle$. 
In Section \ref{sec:MpSymGad}, we characterize the permutable matchgate signatures by constructing a corresponding gadget consisting of symmetric matchgate signatures.
In Section \ref{sec:Redctosym}, for each asymmetric case of $f\in\mathcal{F}\cap(\mathscr{M}-\mathscr{A})$, we reduce a corresponding symmetric case to it by gadget construction in the setting of $\#R_3$-CSP.

\subsection{Normalize the matchgate}\label{sec:MGnormal}

A matchgate signature is said to be \textit{non-trivial} if it does not remain constant at 0. We say $f$ is a \textit{normalized signature} if $f(0\dots0)=1$. For a normalized matchgate signature $f$ of arity $n$ and distinct $1\le b_1,\dots,b_k\le n, 0\le k\le n$, we define $F(b_1\dots b_k)=f(\alpha)$ where $\alpha_{b_1}=\dots=\alpha_{b_k}=1$ and $\alpha_i=0$ for each $1\le i\le n, i\neq b_1,\dots,b_k$. For example, if the arity of $f$ is 4, then $F(24)=f(0101)$ while $F()=f(0000)=1$. We denote $F$ as the \textit{index expression} of $f$, and we also say $F$ is a normalized matchgate signature without causing ambiguity. 

This section presents the relationship between non-trivial matchgate signatures and normalized matchgate signatures, together with a property that normalized matchgate signatures have. The results in this section can be seen as a partial restatement of the results in \cite{cai2013matchgates}.

\begin{lemma}
    Each non-trivial matchgate signature $g$ of arity $n$ can be realized by a normalized matchgate signature $f$ of arity $n$ and $O(n)$ $[0,1,0]$ signatures, up to a constant factor.
    \label{lem:normalize}
\end{lemma}
\begin{proof}
    Since $g$ is non-trivial, there exists $\beta\in \{0,1\}^n$ satisfying $g(\beta)\neq 0$. For each $\alpha\in \{0,1\}^n$, we let $f(\alpha)=g(\alpha\oplus\beta)/g(\beta)$. It can be verified that $f(0\dots0)=1$ and $f$ also satisfy MGI, which implies that $f$ is a normalized matchgate signature. Then we can use the following gadget to realize $\frac{1}{g(\beta)}g$: for each $1\le i\le n$ satisfying $\beta_i=1$, we connect a $[0,1,0]$ signature to the $i$th variable of $f$.
\end{proof}

We denote $f$ as the \textit{normalization} of $g$ in the above lemma. Suppose $k$ is an integer, $b_1,\dots,b_{2k}\in \mathbb{N}^+$ and $b_1<\dots<b_{2k}$. $S=\{b_1,\dots,b_{2k}\}$ is said to be an \textit{index set of size $2k$}. A \textit{pairing} $M$ of $S$ is a partition of $S$ whose components contain exactly 2 elements. In other words, $M$ can be seen as a perfect matching on the graph $(S,S\times S)$. In addition, suppose $a,b,c,d\in S$ and $a<b<c<d$. If $M$ is a pairing of $S$ and $ac,bd\in M$, then $(ac,bd)$ is said to be a \textit{crossing} in $M$. We use $c(M)$ to denote the number of crossings in $M$.

We also give a visualization of the definitions above. We draw all elements in $S$ on a circle in a sequential order. Given a pairing $M$, we draw a straight line between the two elements in each pair belonging to $M$. A crossing in $M$ is formed if and only if two of the straight lines form a crossing. See Figure \ref{fig:pairing} for an example.
\begin{figure}
            \centering
            \includegraphics[height=0.2\textheight]{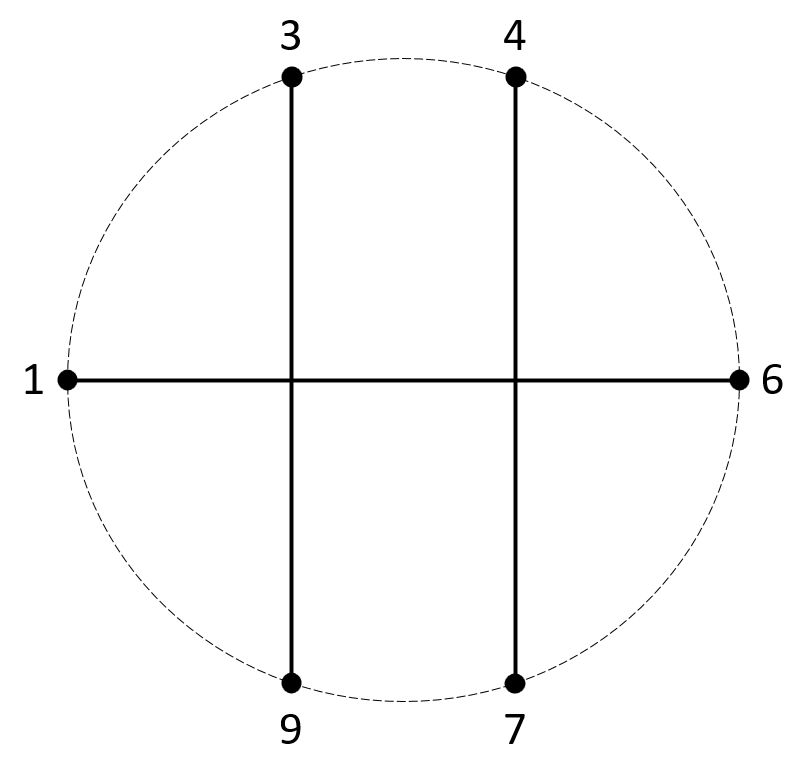}
            \caption{A visualization of the pairing $M=\{(1,6),(3,9),(4,7)\}$.} 
            \label{fig:pairing}
        \end{figure}
By the construction of the universal matchgate in \cite{cai2013matchgates}, we have the following lemma.

\begin{lemma}
    Suppose $F$ is a normalized signature of arity $n$ and of even parity. Then $F$ is a matchgate signature if and only if for each integer $0\le k\le n/2$ and distinct $1\le b_1,\dots,b_{2k}\le n$, 
    
    $$F(b_1\dots b_{2k})=\sum_{M:M\text{ is a pairing of }\{b_1,\dots ,b_{2k}\}}{(-1)^{c(M)}\prod_{b_ib_j\in M}F(b_ib_j)}$$
    \label{lem:MGvalue}
\end{lemma}

\subsection{Permutation Check}\label{sec:PermCheck}
In this section, we show that given a normalized matchgate signature $F$ of arity $n$ and a permutation $\pi$, we can decide whether $F_{\pi}$ is a matchgate signature in polynomial time in $n$. To be precise, we only need to check whether $F_{\pi}$ satisfy all the properties in Lemma \ref{lem:MGvalue} restricting to $k=2$.
\begin{theorem}\label{lem:4andpiMG}
    Suppose $F$ is a normalized matchgate signature and $\pi$ is a permutation.
    If for each $1\le a<b<c<d\le n$, $F_{\pi}(abcd)=F_{\pi}(ab)F_{\pi}(cd)-F_{\pi}(ac)F_{\pi}(bd)+F_{\pi}(ad)F_{\pi}(bc)$, then $F_{\pi}$ is also a matchgate signature. 
\end{theorem}

 Before proving the above lemma, we introduce the concept of the partition of a pairing $M$. Given a set $M$, a \textit{partition} of $M$ are two subsets $M_1,M_2\subseteq M$ satisfying $M=M_1\cup M_2,M_1\cap M_2=\emptyset$. If $M$ is a pairing, we also use $c(M_1,M_2)$ to denote the number of crossings formed by a pair in $M_1$ and another one in $M_2$. Further, if $M_1=\{ab\}$, we also denote $c(M_1,M_2)$ as $c(ab,M_2)$. Obviously,
    $c(M)=c(M_1)+c(M_2)+c(M_1,M_2)$.
    Besides, the following lemma shows that for any pairing $M$, $c(M_1,M_2) \mod 2$ only depends on the partition of the index, irrespective of the specific form of the pairs within $M_1$ or $M_2$. 
    \begin{lemma}
       Let $S$ be an index set and $S_1,S_2$ form a partition of $S$. Then for any pairing $M_1$ of $S_1$ and $M_2$ of $S_2$, $c(M_1,M_2)\mod 2$ is independent of $M_1$ and $M_2$ and only depends on $S_1$ and $S_2$.
       \label{lem:partcross}
    \end{lemma}
    \begin{proof}
    Suppose $S=\{b_1,...,b_{2k}\}$ is of size $2k$ satisfying $b_1<\dots<b_{2k}$. Without loss of generality, we map $b_i$ to $i$ for each $1\le i\le 2k$ and consequently $S$ is mapped into $\{1,\dots,2k\}$.
    
        Fixing $M_1$, we consider the value of $c(ab,M_2)$, where $ab\in M_1$. Let $S_2'=\{c\in S_2\mid a<c<b\}$ and each crossing between $ab$ and $M_2$ is formed by $ab$ and $cd\in M_2$ satisfying $c\in S_2'$ and $d\in S_2-S_2'$. For each $c\in S_2'$, $c$ either pairs with another $e\in S_2'$, or pairs with $d\in S_2-S_2'$. This gives us 
        
         $$c(ab,M_2)=|\{c\in S_2'\mid cd\in M_2,d\in S_2-S_2'\}|=|S_2'|-|\{c\in S_2'\mid ce\in M_2,e\in S_2'\}|.$$
         Furthermore, $|\{c\in S_2'\mid ce\in M_2,e\in S_2'\}|$ is an even number. Consequently, we have $c(ab,M_2)\mod 2=|S_2'|\mod 2$.
       
       As $c(M_1,M_2)\mod 2=\sum_{ab\in M_1} c(ab,M_2) \mod 2$, $c(M_1,M_2)\mod 2$ is independent of the form of $M_2$. Similarly, by fixing $M_2$ we can prove that $c(M_1,M_2)\mod 2$ is independent of the form of $M_1$. The lemma is proved as the value of $c(M_1,M_2)\mod 2$ remains unchanged when $M_1$ and $M_2$ are modified arbitrarily.
    \end{proof}
     By Lemma \ref{lem:partcross}, now we can also denote $c(M_1,M_2)\mod 2$ as $c(S_1,S_2)\mod 2$. Now we are ready to prove Theorem \ref{lem:4andpiMG}. 
     The main objective of the proof is to demonstrate the equality between two summations, given by $F$ and $F_\pi$ through Lemma \ref{lem:MGvalue}. The main idea is to partition the summations into discrete components, identify the numerical relationships between these components, and then prove the equality within each component.

\begin{proof}[Proof of Theorem \ref{lem:4andpiMG}]
     By Lemma \ref{lem:MGvalue}, we only need to prove that for each integer $0\le k\le n/2$ and distinct $1\le b_1,\dots,b_{2k}\le n$, 
     
    $$F_\pi(b_1\dots b_{2k})=\sum_{M:M\text{ is a pairing of }\{b_1,\dots ,b_{2k}\}}{(-1)^{c(M)}\prod_{b_ib_j\in M}F_\pi(b_ib_j)}$$

    For convenience, we use $a_i$ to denote $\pi(b_i)$ for each $1\le i\le 2k$ in the following proof. As $F$ is a matchgate signature, we already have that
    
    $$F_{\pi}(b_1\dots b_{2k})
        =F(a_1\dots a_{2k})
        = \sum_{M_\pi:M_\pi\text{ is a pairing of }\{a_1,\dots ,a_{2k}\}}(-1)^{c(M_\pi)}\prod_{a_ia_j\in M_\pi}F(a_ia_j)
    $$
    
   
   In particular, $F_\pi(b_ib_j)=F(a_ia_j)$, and we only need to show that for each integer $0\le k\le n/2$ and distinct $1\le b_1,\dots,b_{2k}\le n$,
   \begin{align*}
       &\sum_{M_\pi:M_\pi\text{ is a pairing of }\{a_1,\dots ,a_{2k}\}}(-1)^{c(M_\pi)}\prod_{a_ia_j\in M_\pi}F(a_ia_j)\\
    =&\sum_{M:M\text{ is a pairing of }\{b_1,\dots ,b_{2k}\}}{(-1)^{c(M)}\prod_{b_ib_j\in M}F(a_ia_j)}
   \end{align*}
    
   We prove this by induction. By the given condition this is true when $k=1,2$. Now suppose this is true for $k=1,\dots ,p-1$, and we focus on the situation that $k=p$. 
    We use $a_1'<\dots <a_{2p}'$ to rename $a_1,\dots ,a_{2p}$, and let $\tau$ be a 1-to-1 mapping such that $a_i=a_{\tau(i)}'$. 
If $i$ and $\tau(i)$ are both even or odd, we call $i$ a stable point, otherwise we call $i$ a mutual point. For a pair $b_ib_j$ or $a_ia_j$, if both $i,j$ are stable or both are mutual, we call it a SS pair or a MM pair respectively. Otherwise, we call it a SM pair.
    We also remark that the number of mutual points must be even. Otherwise, the number of odd positions before and that after the permutation can not be equal, which forms a contradiction. This also implies that the number of stable points must be even.
    
    Now, we make a partition over all the pairings. We may assume that the number of stable points $2s$ is greater than, or equal to, the number of mutual point $2p-2s$. This is because $F$ remains a matchgate signature under the permutation $(234\dots (2p)1)$, and this permutation makes all the stable points become mutual and all the mutual points become stable. We use $\mathcal{M}_g$ to denote the set of all pairings of $\{b_1,\dots ,b_{2p}\}$ that contain exactly $g$ SS pairs and $\mathcal{M}_g'$ to denote the set of all pairings of $\{a_1,\dots ,a_{2p}\}$ that contain exactly $g$ SS pairs. We let $m_g=\sum_{M\in \mathcal{M}_g}(-1)^{c(M)}\prod_{b_ib_j\in M}F(a_ia_j)$ and $m_g'=\sum_{M\in \mathcal{M}_g'}(-1)^{c(M_\pi)}\prod_{a_ia_j\in M_\pi}F(a_ia_j)$. Using these notations, it is sufficient to demonstrate that
    \[\sum_{g=0}^s m_g=\sum_{g=0}^s m_g'\]
    
   If $S$ is a subset of $\{b_1,\dots,b_{2p}\}$, we use $\pi(S)$ to denote the set $\{a_i|b_i\in S\}$. Let $S_1,S_2$ be a partition of $\{b_1,\dots,b_{2p}\}$ and $S_1$ only contains stable points. We emphasize that except for the appearance of specific remark, the restriction that $S_1$ only contains stable points always exists in the following analysis, even though it is omitted in some summations using $S_1$ as a symbol. By Lemma \ref{lem:partcross}, we know that for arbitrary pairing $M_1$ of $S_1$ and $M_2$ of $S_2$, $c(M_1,M_2)\mod 2$ is a constant. Assume $S_1=\{b_{i_1},\dots,b_{i_{2q}}\}$ satisfying $b_{i_1}<\dots<b_{i_{2q}}$, by setting $M_1=\{b_{i_1}b_{i_2},\dots,b_{i_{2q-1}}b_{i_{2q}}\}$, we have 
    \begin{align*}
        &c(M_1,M_2)\mod 2\\
        =&\sum_{ab\in M_1} c(ab,M_2) \mod 2\\
        =&\sum_{j=1}^q {i_{2j}-i_{2j-1}-1}\mod 2\\
        =&\sum_{j=1}^q {i_{2j}+i_{2j-1}+1}\mod 2\\
        =&q+\sum_{j=1}^{2q} {i_{j}}\mod 2
    \end{align*}

    Similarly, we have $c(\pi(S_1),\pi(S_2))=q+\sum_{j=1}^{2q} {\tau(i_{j})}\mod 2$. Since all vertices in $S_1$ are stable, for each $1\le j\le 2q$, $i_j\mod 2=\tau(i_j)\mod 2$ and as a result $c(S_1,S_2)\mod 2=c(\pi(S_1),\pi(S_2))\mod 2$. 
    
    Now we define $\Omega(S_1)$ to be the summation over all pairings that can be separated into a pairing of $S_1$ and a pairing of $S_2$. $\Omega(S_1)$ can be written as the following form.
    \begin{align*}
        \Omega(S_1)=&\sum_{M:M=M_1\cup M_2, M_1,M_2\text{ is a pairing of } S_1,S_2\text{ respectively}}{(-1)^{c(M)}\prod_{b_ib_j\in M}F(b_ib_j)}\\
        =&\sum_{M:\dots\text{(same as the last line)}}{(-1)^{c(M_1)+c(M_2)+c(M_1,M_2)}\prod_{b_ib_j \in M_1} F(b_ib_j) \prod_{b_ib_j\in M_2} F(b_ib_j)}\\
        =&(-1)^{c(S_1,S_2)}(\sum_{M_1:M_1\text{ is a pairing of } S_1}{(-1)^{c(M_1)}\cdot\prod_{b_ib_j\in M_1}F(a_ia_j)})\\
        &\text{  \  \ \  \ \ }\cdot (\sum_{M_2:M_2\text{ is a pairing of } S_2}{(-1)^{c(M_2)}\prod_{b_ib_j\in M_2}F(a_ia_j)})\\
    \end{align*}
  
    If $S_1,S_2\neq \emptyset$, from the induction condition and the fact that $c(S_1,S_2)\mod 2=c(\pi(S_1),\pi(S_2))\mod 2$ we have
    \begin{align*}
        \Omega(S_1)=&(-1)^{c(S_1,S_2)}(\sum_{M:M\text{ is a pairing of } S_1}{(-1)^{c(M)}\cdot\prod_{b_ib_j\in M}F(a_ia_j)})\\
        &\text{  \  \ \  \ \ }\cdot (\sum_{N:N\text{ is a pairing of } S_2}{(-1)^{c(N)}\prod_{b_ib_j\in N}F(a_ia_j)})\\
        =&(-1)^{c(\pi(S_1),\pi(S_2))}(\sum_{M_\pi:M_\pi\text{ is a pairing of }\pi(S_1)}(-1)^{c(M_\pi)}\cdot\prod_{a_ia_j\in M_\pi}F(a_ia_j))\\
        &\text{  \  \ \  \ \ }\cdot (\sum_{N_\pi:N_\pi\text{ is a pairing of }\pi(S_2)}(-1)^{c(N_\pi)}\prod_{a_ia_j\in N_\pi}F(a_ia_j))\\
        =&\Omega(\pi(S_1))
    \end{align*}
    
    If we consider the summation $\sum_{|S_1|=2}\Omega(S_1)$, we can find that each pairing in $\mathcal{M}_g$ appears in the summation $g$ times. Consequently, we have\footnote{Here we also omit the restriction to $S_1$ that it is composed of stable points.\label{ft:omit}}
    \[\sum_{g=1}^s gm_g=\sum_{|S_1|=2}\Omega(S_1)=\sum_{|S_1|=2}\Omega(\pi(S_1))=\sum_{g=1}^s gm_g'\]  
    If $s=p$, all the points are stable points and $m_g=m_g'=0$ for each $0\le g\le s-1$. Then by this equation we also have $m_s=m_s'$ and complete the proof.
    
    Now we may assume $s<p$. Similarly, for each $1\le t\le s$, we can also get that\footref{ft:omit}
    \[\sum_{g=t}^s \binom{g}{t} m_g=\sum_{|S_1|=2t}\Omega(S_1)=\sum_{|S_1|=2t}\Omega(\pi(S_1))=\sum_{g=t}^s \binom{g}{t}m_g'\]
     We remark that the equation still holds when $t=s$ since $s<p$ and consequently $S_2\neq \emptyset$. These $s$ equations form a system whose coefficient matrix is an upper triangular matrix. As a result, this system implies that $m_g=m_g'$ for $g=1,\dots ,s$.
    
    If the number of stable points is strictly greater than the number of mutual point, then $\mathcal{M}_0=\emptyset$ and $m_0=m_0'=0$, which completes the proof.     
    Otherwise, $s=p/2$. We take two stable points and two mutual points as $S_1=\{b_{i_1},\dots,b_{i_4}\}$. We remark that although this selection of $S_1$ violates the restriction described above, $c(S_1,S_2)\mod 2=2+\sum_{j=1}^4 i_j\mod 2=2+2+\sum_{j=1}^4 \tau(i_j)\mod 2=c(\pi(S_1),\pi(S_2))$ still holds and the above analysis still works. Consequently, we have $\Omega(S_1)=\Omega(\pi(S_1))$. For a pairing in $\mathcal{M}_0$, each pair is a SM pair. Therefore, each pairing in $\mathcal{M}_0$ is counted $\binom{2s}{2}$ times in $\sum_{S_1}\Omega(S_1)$\footnote{Here the omitted restriction is that $S_1$ is composed of two stable points and two mutual points.}. For a pairing in $\mathcal{M}_g$, it is composed of $2s-2g$ SM pairs, $g$ SS pairs and $g$ MM pairs. Therefore, each pairing in $\mathcal{M}_g$  is counted $\binom{2s-2g}{2}+g^2$ times in $\sum_{S_1}\Omega(S_1)$. Consequently,
    \[\sum_{g=0}^s (\binom{2s-2g}{2}+g^2)m_g=\sum_{S_1}\Omega(S_1)=\sum_{S_1}\Omega(\pi(S_1))=\sum_{g=0}^s (\binom{2s-2g}{2}+g^2)m_g'\]
    As $m_g=m_g'$ for each $1\le g\le s$, $m_0=m_0'$.
\end{proof}

\subsection{Permutable matchgate signatures}\label{sec:easyMP}
\begin{definition}\label{def:MP}
    Suppose $f$ is a signature of arity $n$. For a permutation $\pi\in S_n$, we use $f_\pi$ to denote the signature
    
$$f_\pi(x_1,\dots,x_n)=f(\pi(x_1),\dots,\pi(x_n))$$
If for each $\pi\in S_n$, $f_\pi$ is a matchgate signature, we say $f$ is a permutable matchgate signature.

We use $\mathscr{M}_{P}$ to denote the set of all the permutable matchgate signatures.
\end{definition}

The following property can be easily verified.
  \begin{lemma}
        If $f$ is a permutable matchgate signature of arity $k$, then for arbitrary $1\le p\le k$ and $\alpha\in \{0,1\}^p$, $f^\alpha$ is also a permutable matchgate signature.
        \label{lemPin}
    \end{lemma}

The following theorem shows that $\mathscr{M}_{P}$ is closely related to the dichotomy for $\text{\#CSP}\langle \mathcal{PL}\rangle$ and $\#R_D\text{-CSP}\langle \mathcal{PL}\rangle$, which can be derived from Theorem \ref{plCSP} and \ref{plRDCSP} respectively.

\begin{theorem} \label{thmCSPpl}
    If $\mathcal{F}\subseteq \mathscr{A}$ or $\mathcal{F}\subseteq \mathscr{P}$ or $\mathcal{F}\subseteq \widehat{\mathscr{M}_P}$, $\text{\#CSP}(\mathcal{F})\langle \mathcal{PL}\rangle$ is computable in polynomial time; otherwise it is $\text{\#P}$-hard.

    If $\mathcal{F}\subseteq \mathscr{A}$ or $\mathcal{F}\subseteq \mathscr{P}$ or $\mathcal{F}\subseteq \widehat{\mathscr{M}_P}$, $\#R_D\text{-CSP}\langle \mathcal{PL}\rangle$ is computable in polynomial time; otherwise it is $\text{\#P}$-hard.
\end{theorem}
\begin{proof}
    Let $\mathcal{F'}=\{f_\pi|f\in \mathcal{F},\text{arity}(f)=n, \pi\in S_{n}\}$. By Remark \ref{remembd}, we have 
    
    $$\text{\#CSP}(\mathcal{F})\langle \mathcal{PL}\rangle\equiv_T\text{Pl-\#CSP}(\mathcal{F'})$$
    
    By Definitions \ref{defA} and \ref{defP}, $\mathcal{F'}\subseteq \mathscr{A}$ or $\mathscr{P}$ if and only if $\mathcal{F}\subseteq \mathscr{A}$ or $\mathscr{P}$ respectively. Furthermore,  $\mathcal{F'}\subseteq \widehat{\mathscr{M}}$ if and only if $\mathcal{F}\subseteq \widehat{\mathscr{M}_P}$ by Definition \ref{def:MP} and Theorem \ref{thmHT}. Consequently, we are done by replacing the tractable criteria for $\mathcal{F'}$ in Theorem \ref{plCSP} with those for $\mathcal{F}$. The same argument holds for $\#R_D\text{-CSP}\langle \mathcal{PL}\rangle$ as well.
\end{proof}

\subsection{Characterizations of permutable matchgate signatures}\label{sec:MpSymGad}
By Theorem \ref{lem:4andpiMG}, we can use the following property to characterize the permutable matchgate signature.

\begin{corollary}
    Suppose $F$ is a normalized matchgate signature of arity $n$. Then $F$ is a permutable matchgate signature if and only if for each $1\le a<b<c<d
    \le n$, $F(ab)F(cd)=F(ac)F(bd)=F(ad)F(bc)$.
    \label{coro:4MPequal}
\end{corollary}

\begin{proof}
    By Theorem \ref{lem:4andpiMG}, $F$ is a permutable matchgate signature if and only if for any permutation $\pi\in S(n)$ and integers $1\le a<b<c<d
    \le n$, $F_{\pi}(abcd)=F_{\pi}(ab)F_{\pi}(cd)-F_{\pi}(ac)F_{\pi}(bd)+F_{\pi}(ad)F_{\pi}(bc)$. Notice that $F(abcd)=F_{\pi}(\pi^{-1}(a)\pi^{-1}(b)\pi^{-1}(c)\pi^{-1}(d))$ also holds for arbitrary $\pi$, hence by taking different $\pi$, we have the following equation.
    
    \begin{align*}
        F(abcd)=&F(ab)F(cd)+F(ac)F(bd)-F(ad)F(bc)\\
        =&F(ab)F(cd)-F(ac)F(bd)+F(ad)F(bc)\\
        =&-F(ab)F(cd)+F(ac)F(bd)+F(ad)F(bc)
    \end{align*}
    which implies $F(ab)F(cd)=F(ac)F(bd)=F(ad)F(bc)$.

    On the other hand, if $F(ab)F(cd)=F(ac)F(bd)=F(ad)F(bc)$ holds for arbitrary $1\le a<b<c<d
    \le n$, then for any $\pi$, $F_{\pi}(abcd)=F(\pi(a)\pi(b))F(\pi(c)\pi(d))=F_{\pi}(ab)F_{\pi}(cd)-F_{\pi}(ac)F_{\pi}(bd)+F_{\pi}(ad)F_{\pi}(bc)$.
\end{proof}

Using the description above, we are able to classify and characterize the normalized permutable matchgate signatures in the following way.

\begin{lemma}
    For a normalized permutable matchgate signature $F$ of arity $n\ge 4$, one of the following holds:
    \begin{enumerate}
        \item (Pinning type) For any distinct $1\le a,b\le n$, we have $F(ab)=0$.
        \item (Parity type 1) There exist distinct $1\le a,b,c,d\le n$, such that $F(ab)F(cd)\neq 0$.
        \item (Parity type 2) There exist distinct $1\le a,b,c\le n$, such that $F(ab)F(ac)F(bc)\neq 0$.
        \item (Matching type) There exist distinct $1\le a,b,c,d\le n$, such that $F(bc),F(bd),F(cd)=0$, but $F(ab)\neq 0$.
    \end{enumerate}
\end{lemma}
\begin{proof}
    Suppose otherwise. As $F$ is not of Pinning type, there exists $1\le a,b\le n$ such that $F(ab)\neq 0$. Then for any distinct $1\le c,d \le n$ and $c,d\neq a,b$, if $F(cd)\neq 0$, $F$ is of Parity type 1, so we may assume $F(cd)=0$. 

    As $F$ is not of Parity type 1, $F(ac)F(bd)=F(ad)F(bc)=0$. As $F$ is not of Parity type 2, $F(ac)F(bc)=0$ and $F(ad)F(bd)=0$. Consequently, either $F(ac),F(ad)=0$, or $F(bc),F(bd)=0$. In either case $F$ is of Matching type, a contradiction.
\end{proof}
If $F$ is of Parity type 1, then there exist distinct $1\le a,b,c,d\le n$, such that $F(ac)F(bd)=F(ad)F(bc)=F(ab)F(cd)\neq 0$, and we have $F(ab)F(ac)F(bc)\neq 0$, indicating that $F$ is also of Parity type 2. Consequently, Parity type 1 and 2 can be concluded into a single type, denoted as \textit{Parity type}, in which each signature satisfy the condition of Parity type 2.
\begin{lemma}
    If $F$ is a normalized permutable matchgate signature of arity $n$ of Parity type, then there exist a function $G:\{1,\dots,n\}\to \mathbb{C}$ such that for any distinct $1\le a,b \le n$, $F(ab)=G(a)G(b)$.
    \label{lem:paritytypeproperty}
\end{lemma}

\begin{proof}
    As $F$ is of Parity type, there exist distinct $1\le a,b,c\le n$, such that $F(ab)F(ac)F(bc)\neq 0$.
    To avoid ambiguity, we use $\sqrt{F(ab)}$ to denote a specific complex number $r_{ab}$ satisfying $r_{ab}^2=F(ab)$. We also use $\sqrt{F(ab)F(ac)/F(bc)}$ to denote the complex number $\sqrt{F(ab)}\sqrt{F(ac)}/\sqrt{F(bc)}$. Let $G(a)=\sqrt{F(ab)F(ac)/F(bc)}\neq 0$, $G(b)=\sqrt{F(ab)F(bc)/F(ac)}$ and  $G(c)=\sqrt{F(ac)F(bc)/F(ab)}$.
    For each $1\le d\le n, d\neq a,b,c$, let $G(d)=F(ad)/G(a)$. Since $F(ab)F(ac)F(bc)\neq 0$, $G$ is well-defined.

    It is easy to verify that $F(ab)=G(a)G(b),F(ac)=G(a)G(c),F(bc)=G(b)G(c)$. For each $1\le d\le n, d\neq a,b,c$,
    
    $$G(a)G(d)=F(ad)$$
    $$G(b)G(d)=\frac{F(ad)G(b)G(c)}{G(a)G(c)}=\frac{F(ad)F(bc)}{F(ac)}=\frac{F(bd)F(ac)}{F(ac)}=F(bd)$$
  $$G(c)G(d)=\frac{F(ad)G(c)G(b)}{G(a)G(b)}=\frac{F(ad)F(bc)}{F(ab)}=\frac{F(cd)F(ab)}{F(ab)}=F(cd)$$

    For any distinct $1\le d,e\le n$ satisfying $d,e\neq a,b,c$,

    $$G(d)G(e)=\frac{F(ad)F(be)}{G(a)G(b)}=\frac{F(de)F(ab)}{F(ab)}=F(de)$$
\end{proof}
Now we analyze the property of normalized permutable matchgate signatures of Matching type.
\begin{lemma}
     Suppose $F$ is a normalized permutable matchgate signature of arity $n$ of Matching type. If $F$ is not of Parity type, then there exists an integer $1\le x\le n$ such that for any distinct $1\le s,t\le n$, if $s,t\neq x$, $F(st)=0$.
     Furthermore, if $2\le k\le n/2$ and $1\le b_1<\dots<b_{2k}\le n$, then $F(b_1\dots b_{2k})=0$.
     \label{lem:matchingtypeproperty}
\end{lemma}
\begin{proof}
Suppose otherwise. Since $F$ is of Matching type, there exist distinct $1\le a,b\le n$ such that $F(ab)\neq 0$. There are three possible cases.
\begin{enumerate}
    \item There exist distinct $1\le s,t\le n, s,t\neq a,b$ such that $F(st)\neq 0$. It is obvious that $a,b,s,t$ can serve as a certificate that $F$ is of Parity type 1.
    \item There exist distinct $1\le s,t\le n, s,t\neq a,b$ such that $F(as)F(bt)\neq 0$. Again, $a,b,s,t$ can serve as a certificate that $F$ is of Parity type 1.
    \item There exist $1\le s\le n, s\neq a,b$ such that $F(as)F(bs)\neq 0$. In this case, $a,b,s$ can serve as a certificate that $F$ is of Parity type 2.
\end{enumerate} 

In each case, $F$ is of Parity type, which is a contradiction.

By Lemma \ref{lem:MGvalue}, if $2\le k\le n/2$ and $1\le b_1<\dots<b_{2k}\le n$, then $F(b_1\dots b_{2k})=\sum_{M:M\text{ is a pairing of }\{b_1,\dots ,b_{2k}\}}{(-1)^{c(M)}\prod_{b_ib_j\in M}F(b_ib_j)}=0$, since for each $M$ of size greater than 2 there always exists $st\in M$ such that $s,t\neq x$. 
\end{proof}

\begin{theorem}\label{lem:SymGadRealize}
    Each permutable matchgate signature $F'$ of arity $n$ can be realized by a symmetric matchgate signature $g$ and $O(n)$ symmetric binary matchgate signatures up to a constant factor in the following way. 
    \begin{enumerate}
        \item If $F'$ is of Pinning type after normalization, then $h=[1,0,\dots,0]$ and $O(n)$ symmetric binary signatures $[0,1,0]$ can realize $F'$.
        \item If $F'$ is of Parity type after normalization, then $h=[1,0,1,0,\dots]$ or $[0,1,0,1,\dots]$ and the $O(n)$ symmetric binary signatures with the form $[1,0,y]$ or $[0,0,1]$ can realize $F'$.
        \item If $F'$ is of Matching type, not of Pinning type and not of Parity type after normalization, then $h=[0,1,0,0,\dots]$ and the $O(n)$ symmetric binary signatures with the form $[1,0,y]$ or $[0,1,0]$ can realize $F'$.
    \end{enumerate}
\end{theorem}
\begin{proof}
   Suppose $f$ (or the index expression $F$ of $f$) is a normalization of $F'$. For each case, we first realize $F$, then we analyze $F'$ through Lemma \ref{lem:normalize}. 

   If $F$ is of Pinning type, by Lemma \ref{lem:MGvalue} $f=[1,0,0,\dots,0]$. By Lemma \ref{lem:normalize}, $F'$ can be realized by connecting a $[0,1,0]$ signature to some variables of $F$. 
   
   If $F$ is of Parity type, by Lemma \ref{lem:paritytypeproperty} there exists a function $G$ such that for any distinct $1\le a,b\le n$, $F(ab)=G(a)G(b)$. Then we can use the following gadget to realize $F$: for each $1\le a\le n$, we connect a $[1,0,G(a)]$ signature to the $a$th variable of $h'=[1,0,1,0,\dots]$. By Lemma \ref{lem:normalize}, $F'$ can be realized by connecting a $[0,1,0]$ signature to some variables of $F$. 
   
   Now for some variables of $h'$, they are connected to a $[1,0,G(a)]$ signature, then a $[0,1,0]$ signature, where $a$ is the index of the variable. For each such variable, we replace the $[1,0,G(a)]$ with a $[0,1,0]$ and the $[0,1,0]$ with a $[G(a),0,1]$ respectively. The signature of the gadget remains the same after the replacement as  $\begin{pmatrix}
          1&0\\
          0&G(a)
      \end{pmatrix}\begin{pmatrix}
          0&1\\
          1&0
      \end{pmatrix}=\begin{pmatrix}
          0&1\\
          G(a)&0
      \end{pmatrix}=\begin{pmatrix}
          0&1\\
          1&0
      \end{pmatrix}\begin{pmatrix}
          G(a)&0\\
          0&1
      \end{pmatrix}$.
   
   Now consider the gadget formed by $h'=[1,0,1,0,\dots]$ and all the $[0,1,0]$ connecting to it. If there is an odd number of $[0,1,0]$ signatures, the signature of the gadget is $h=[0,1,0,1,\dots]$. If there is an even number of $[0,1,0]$ signatures, the signature of the gadget is $h=[1,0,1,0,\dots]$. For each $1\le a\le n$, the binary signature connecting to the $a$th variable of $h$ is either $[1,0,G(a)]$ or $[G(a),0,1]$, which has the form $[1,0,y]$ or $[0,0,1]$ up to a constant factor.

   If $F$ is of Matching type, not of Pinning type and not of Parity type, then by Lemma \ref{lem:matchingtypeproperty} there exists $1\le x\le n$ such that the following holds: for any distinct $1\le a,b\le n$, if $F(ab)\neq 0$, then $x\in \{a,b\}$. Let $G(x)=1$ and $G(a)=F(ax)$ for each $1\le a\le n, a\neq x$. It can be verified that the following gadget realize $F$: for each $1\le a\le n, a\neq x$, we connect a $[1,0,G(a)]$ signature to the $a$th variable of $h=[0,1,0,0,\dots]$; we also connect a $[0,1,0]$ signature to the $x$th variable. By Lemma \ref{lem:normalize}, $F'$ can be realized by connecting a $[0,1,0]$ signature to some variables of $F$, which completes the proof. Besides, if there are two $[0,1,0]$ connecting to the $x$th variable of $h$, we can remove them  without changing the signature of the gadget for future convenience. 
\end{proof}

We denote the gadget in the above lemma as the \textit{star gadget} $ST_F$ of $F$, and $h$ as the \textit{central signature} $h_F$ of $ST_F$.
For each $1\le a\le n$, all the binary signatures connecting to the $a$th variable of $h$ in a line form a gadget, and the signature of the gadget is denoted as the \textit{$a$th edge signature}.

\subsection{Realize symmetric signatures}\label{sec:Redctosym}
In this section, we prove the following lemma by gadget construction.  
\begin{theorem} \label{lem:AsymtoSym}
    For each signature $F\in \mathscr{M}_P-\mathscr{A}$ of arity $n$, a symmetric signature $g\in \mathscr{M}-\mathscr{A}$ can be realized by $\{F\}\mid\{[1,0],[1,0,1],[1,0,1,0]\}$ as a planar left-side gadget.
\end{theorem}
  We remark that $[1,0],[1,0,1],[1,0,1,0]\in \mathscr{M}$ and $\mathscr{M}$ is closed under gadget construction. Hence when proving the above lemma, we only need to ensure that the obtained $g$ is symmetric and does not belong to $\mathscr{A}$. 
  
  Besides, such $g$ must have one of the following forms:
\begin{lemma}[\cite{guo2020complexity}]
    Suppose $g\in \mathscr{M}-\mathscr{A}$ and is symmetric. Then $g$ has one of the following forms.
    \begin{enumerate}
        \item $[0,1,0,...,0]_k,k\ge3$;
        \item $[0,...,0,1,0]_k,k\ge3$;
        \item $[1,0,r], r^4\neq 0,1$;
        \item $[1,0,r,0,r^2,...]_k,k\ge3, r^2\neq 0,1$;
        \item $[0,1,0,r,0,r^2,...]_k,k\ge3, r^2\neq 0,1$.
    \end{enumerate}
    \label{M-Aform}
\end{lemma}

\begin{remark}
     In the following, we always construct the left-side gadget of $\text{Holant}(\{f\} | \{[1,0],[1,0,1],[1,0,1,0]\})$. For future convenience, whenever $uv$ is an edge and both $u$ and $v$ are assigned $f$ in a gadget, we actually mean that we replace $uv$ with $uw,wv$ where $w$ is assigned a $[1,0,1]$ signature in the gadget. Besides, if a gadget is formed by connecting two existing gadgets together, we also automatically replace the connecting edge $uv$ with $uw,wv$, where $w$ is assigned a $[1,0,1]$ signature. These operations would not change the signature of the gadget. Consequently, it can be verified that each obtained gadget always remains a left-side gadget in our following constructions.
     \label{remGadConstruct}
\end{remark}

If $F$ is of Pinning type, we have $F\in \mathscr{A}$ as $[1,0,0,0,\dots],[0,1,0]\in \mathscr{A}$ and $\mathscr{A}$ is closed under gadget construction. Consequently, $F$ is either of Parity type, or of Matching type and not of Parity type.

\subsubsection{Parity type case}
Suppose $F$ is of Parity type. By Theorem \ref{lem:SymGadRealize}, three kinds of binary signatures may connect to the central signature $h_F$ in the star gadget $ST_F$, which are $[1,0,0],[0,0,1]$ and $[1,0,y],y\neq 0$ respectively. Suppose $p$ $[1,0,0]$ and $q$ $[0,0,1]$ are connected to $h_F$. Then we have $F=F'\otimes [1,0]^{\otimes p}\otimes [0,1]^{\otimes q}$ where $F'\in\mathscr{M}_P-\mathscr{A}$ is also of Parity type. 

By assigning $n-p-q$ $[1,0]$ signature to $F'$, we realize $[1,0]^{\otimes p}\otimes [0,1]^{\otimes q}$. By assigning $[1,0]^{\otimes p}\otimes [0,1]^{\otimes q}$ back to $F$ again we can realize the $F'$ signature in a planar way. Consequently, we may assume all the signatures connected to $h_F$ has the form $[1,0,y],y\neq 0$ since we can eliminate all of the $[1,0]$s and $[0,1]$s by gadget construction. We also assume that for each $1\le a \le n$, the $a$th variable of $h_F$ is connected to the binary signature $[1,0,y_a]$.

\textbf{Case 1: $h_F=[1,0,1,0,\dots]$.} If $n=2$, $F$ is already symmetric and we are done. Now we assume $n\ge 3$. For distinct $1\le a,b\le n$, if we connect a $[1,0]$ signature to each variable of $F$ except for the $a$th and the $b$th one, we realize the $[1,0,y_ay_b]$ signature. If $(y_ay_b)^4\neq 1$, $[1,0,y_ay_b]\notin \mathscr{A}$ and we are done. 

Now we suppose for any distinct $1\le a,b\le n$, $(y_ay_b)^4=1$. For any distinct $1\le a,b,c\le n$, we have $y_a^8=(y_ay_b\cdot y_ay_c/y_by_c)^4=1$. If $y_a^4=1$ and $y_b^4=-1$, again we are done since $(y_ay_b)^4=-1\neq 1$. If for each $1\le a\le n$, $y_a^4=1$, then $[1,0,y_a]\in \mathscr{A}$. And since $[1,0,1,0,\dots]\in \mathscr{A}$ and $\mathscr{A}$ is closed under gadget construction, we have $F\in \mathscr{A}$, which is a contradiction. 

 Otherwise, for each $1\le a\le n$, $y_a^4=-1$. Let $\alpha=\sqrt{\ii}=e^{2\pi \ii/8}$, then $y_a=\pm\alpha,\pm\ii\alpha$, and $[1,0,y_ay_b]$ is either $[1,0,\ii]$ or $[1,0,-\ii]$. Notice that each element in $\{\pm\alpha,\pm\ii\alpha\}$ can become $\alpha$ by multiplying $\ii$ or $-\ii$ 0 to 3 times. Consequently, we may connect 0 to 3 $[1,0,\ii]$ or $[1,0,-\ii]$ to each variable of $F$ to get a gadget $ST$, such that after replacing $F$ with the gadget $ST_F$, each edge signature of the obtained star gadget $ST_F'$ is exactly $[1,0,\alpha]$. It can be then verified that, the signature of $ST$, which is also the signature of $ST_F'$, is exactly $[1,0,\ii,0,-1,0,-\ii,0,1,\dots]_{n}$. We are done since $[1,0,\ii,0,-1,0,-\ii,0,1,\dots]_{n}\notin \mathscr{A}$ when $n\ge 3$.

 \textbf{Case 2:$h_F=[0,1,0,1,\dots]$.} If $n=2$, we suppose $F(1)=1,F(2)=y$ despite the presence of a constant factor. In other words, $f=\begin{pmatrix}
          0&y\\
          1&0
      \end{pmatrix}$. Because $F\notin\mathscr{A}$, $y^4\neq 1$. By connecting the first variables of two $F$ signatures to each other, we realize $[1,0,y^2]$. If $y^8=(y^2)^4\neq 1$, we are done. Otherwise, $y=\pm\alpha,\pm\ii\alpha$. By connecting the second variables of three $F$ signatures to $[1,0,1,0]$, as shown in Figure \ref{fig:mod3}, we realize $[0,y^2,0,1]$, which is either $[0,\ii,0,1]\notin \mathscr{A}$ or $[0,-\ii,0,1]\notin \mathscr{A}$.

Otherwise, $n\ge 3$. For distinct $1\le a,b\le n$, if we connect a $[1,0]$ signature to each variable of $F$ except for the $a$th and the $b$th one, we realize the $\neq_2^{y_a,y_b}=\begin{pmatrix}
          0&y_b\\
          y_a&0
      \end{pmatrix}$ signature. By connecting the second variable of $\neq_2^{y_a,y_b}$ to the $a$th variable of $F$, we construct a gadget $GG_{F_a}$ whose signature is $F_a$. 
      
      Now we analyze the properties of $F_a$. By replacing $F$ with the corresponding star gadget $ST_F$ in $GG_{F_a}$, we obtain the gadget $ST_a$ with central signature $h_F=[0,1,0,1,\dots]$. For the $a$th variable of $h_F$, it is connected to a $[1,0,y_a]$ signature, then a $\neq_2^{y_a,y_b}$ signature. We replace the $[1,0,y_a]$ with $[0,1,0]$ and the $\neq_2^{y_a,y_b}$ with $[1,0,y_b]$ respectively. As $\begin{pmatrix}
          0&y_b\\
          y_a&0
      \end{pmatrix}\begin{pmatrix}
          1&0\\
          0&y_a
      \end{pmatrix}=\begin{pmatrix}
          0&y_ay_b\\
          y_a&0
      \end{pmatrix}=y_ay_b\begin{pmatrix}
          1&0\\
          0&1/y_b
      \end{pmatrix}\begin{pmatrix}
          0&1\\
          1&0
      \end{pmatrix}$, the signature of the gadget remains the same up to a constant factor after the replacement.  

      Now consider the gadget formed by $h_F=[0,1,0,1,\dots]$ and the $[0,1,0]$ connecting to the $a$th variable of it. The signature of the gadget is $h=[1,0,1,0,\dots]$, which means that $F_a$ belongs to Parity type Case 1. Consequently, we are done unless for each $1\le c\le n, c\neq a$, $y_c^4=1$. 
      
      By connecting the first variable of $\neq_2^{y_a,y_b}$ to the $b$th variable of $F$, we construct a gadget $GG_{F_b}$ whose signature is $F_b$. Similarly, we are done unless for each $1\le c\le n, c\neq b$, $y_c^4=1$. As a result, the only case left is that for each $1\le c\le n$, $y_c^4=1$. In this case, for each $1\le c\le n$, $[1,0,y_c]\in\mathscr{A}$. Since $[0,1,0,1,\dots]\in\mathscr{A}$ and $\mathscr{A}$ is closed under gadget construction, $F\in \mathscr{A}$, a contradiction.
     \begin{figure*}
	\centering
    \subcaptionbox{\label{fig:mod3}}{ \includegraphics[width=0.3\textwidth]{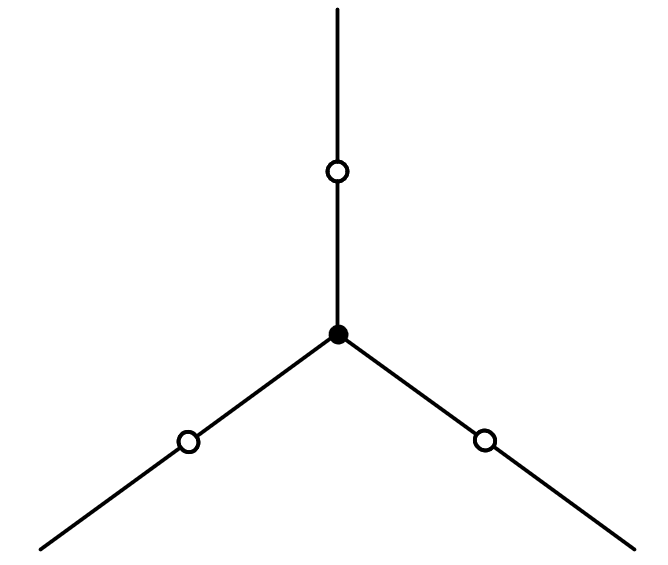}}
			\subcaptionbox{\label{fig:gadhandshake}}{\includegraphics[width=0.45\textwidth]{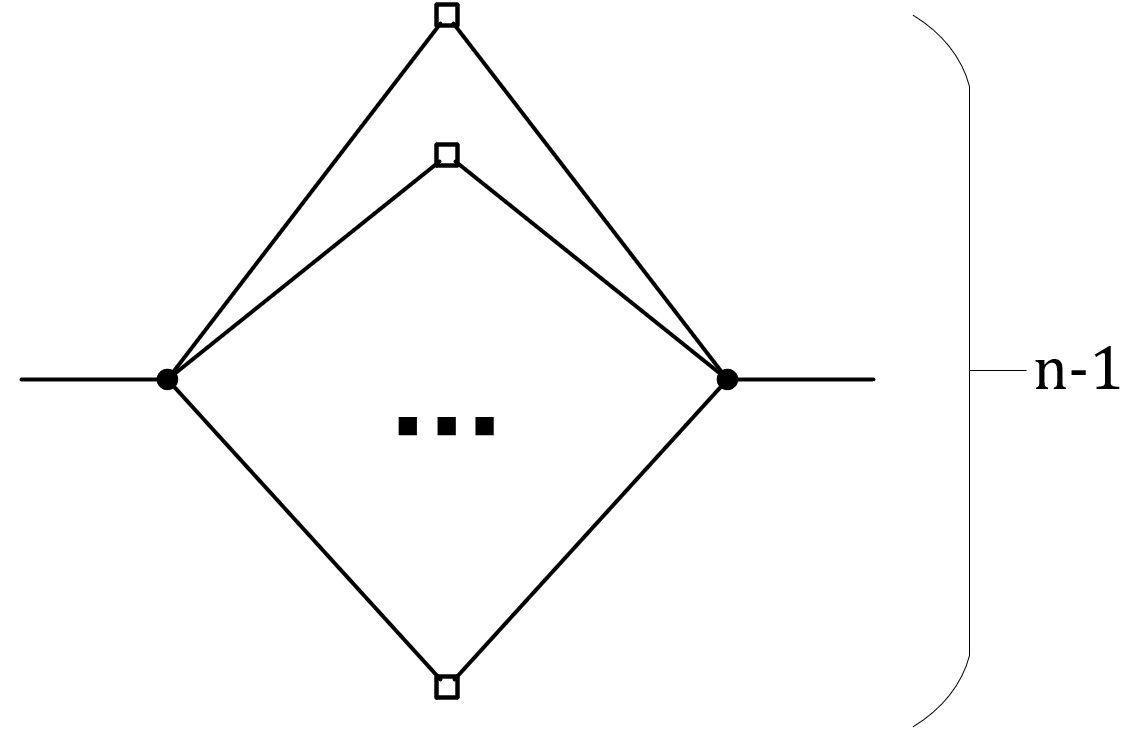}}
	\caption{(a) The construction of a gadget appears in the Case 2 of Parity type. The vertex of degree 3 represented by a solid circle is assigned $[1,0,1,0]$, while each vertex of degree 2 represented by a hollowed circle is assigned $F$ with the second variable connecting to $[1,0,1,0]$. (b) The construction of the generalized mating gadget.}
\end{figure*} 
 \subsubsection{Matching type case}
Suppose $F$ is of Matching type and not of Parity type. By Theorem \ref{lem:SymGadRealize}, two kinds of binary signatures may directly connect to the central signature $h_F$ in the star gadget $ST_F$, which are $[1,0,0]$ and $[1,0,y],y\neq 0$ respectively. Besides, the other variable of each of these signatures might also be connected to a $[0,1,0]$ signature. Suppose $p+q$ $[1,0,0]$ are connected to $h_F$ directly and $q$ of them are further connected to $[0,1,0]$. Then we have $F=F'\otimes [1,0]^{\otimes p}\otimes [0,1]^{\otimes q}$ where $F'\in\mathscr{M}_P-\mathscr{A}$ of arity $k=n-p-q$ is also of Matching type and not of Parity type. 

 When $k=2$, the central signature of $F'$ is $[0,1,0]$, and $F$ is also of Parity type. Therefore, it can be demonstrated that $k\ge 3$. Let $ST_{F'}$ be the star gadget that realize $F'$ as stated in Theorem \ref{lem:SymGadRealize}, and $h_{F'}$ be the central signature. By Theorem \ref{lem:SymGadRealize}, for each $1\le a \le k$, exactly one of the following statements holds. 
 \begin{enumerate}
     \item The $a$th variable of $h_{F'}$ is connected to a $[1,0,y_a],y_a\neq 0$ signature. In this case we denote $a$ as an upright index of $F'$, and the $a$th variable of $F'$ as an upright variable.
     \item The $a$th variable of $h_{F'}$ is connected to a $[1,0,y_a],y_a\neq 0$ signature, then a $[0,1,0]$ signature. In this case we denote $a$ as a reversed index of $F'$, and the $a$th variable of $F'$ as a reversed variable.
 \end{enumerate}
 
 Suppose the number of reversed indexes of $F'$ is $l$. In the following, $F'$ is examined in 4 possible cases based on $l$. In each case, we create generalized mating gadgets defined as follows, which is a generalization of the mating operation in \cite{cai2020realholantodd}. In a \textit{generalized mating gadget}, the variables of $F$ are divided into four parts: \textit{Sum-up} variables, \textit{Fix-to-0} variables, \textit{Fix-to-1} variables and a single \textit{Dangling} variable. We assign $F$ to each vertex of degree $n$ labelled by a solid circle in Figure \ref{fig:gadhandshake} in the following way: for each $1\le a\le n$, if the $a$th variable is the Dangling variable, we let it correspond to the dangling edge. Otherwise we connect it to a vertex $v_a$ of degree 2 labelled by a hollow square. Furthermore, if the $a$th variable is a Sum-up/Fix-to-0/Fix-to-1 variable, we assign a $[1,0,1]$/$[1,0,0]$/$[0,0,1]$ signature to $v_a$, and the gadget become well defined.
 In such constructions, the $p+q$ variables corresponding to $[1,0]$ or $[0,1]$ are always Sum-up variables.
 
 We remark that using $[1,0]$ we may obtain the $[1,0,0]=[1,0]^{\otimes 2}$ signature, and we construct a gadget of $[0,0,1]$ when needed. In the subsequent analysis, we can see that two kinds of generalized mating gadgets would play an important role in each case: the Gadget 1 asks exactly one variable of $F'$ to be Sum-up, while the Gadget 2 asks exactly two variables of $F'$ to be Sum-up. All other upright variables are Fix-to-0 and all other reversed variables are Fix-to-1. 
 
 Also, it should be noted that the subsequent analysis does not include a detailed computation of the values of the signature of each gadget. Nevertheless, readers are encouraged to verify the results of these computations for themselves using the following observation:
 
 Each variable corresponding to $[1,0]$ or $[0,1]$ is a Sum-up variable, and consequently does not contribute to the value of the signature. If $a$ is an upright index, and the $a$th variable of $F'$ is connected to $[1,0]$ or $[1,0,0]$, then the corresponding edge signature of $F'$ does not contribute to the value of the gadget. Besides, the $a$th variable of the central signature $h_{F'}$ are pinned, while $h_{F'}$ remains the form $[0,1,0,0,\dots]$ after this pinning. Similarly, if a reversed variable of $F'$ is connected to $[0,1]$ or $[0,0,1]$, the same statement holds. Furthermore, if a reversed variable is set to be the Sum-up variable in the generalized mating gadget, then the 2 $[0,1,0]$ signatures meet and have no effect to the value of the gadget.
 
\ 

\textbf{Case 1: $l=0$.} Let $1\le a,b,c\le k$ be distinct integers.

 \textbf{Gadget 1:}  Construct a generalized mating gadget. Let the $a$th variable be Dangling, the $b$th variable be Sum-up and all other variables of $F'$ be Fix-to-0. 
 
 By Gadget 1, we realize the signature $[y_b^2,0,y_a^2]$, and we are done unless $(y_a^2)^4=(y_b^2)^4$. Similarly, by replacing $b$ with $c$ in the construction of the above gadget, we are done unless $(y_a^2)^4=(y_c^2)^4$. Now we assume $(y_a^2)^4=(y_b^2)^4=(y_c^2)^4$.

\textbf{Gadget 2:} Construct a generalized mating gadget. Let the $a$th variable be Dangling, the $b$th and the $c$th variable be Sum-up and all other variables of $F'$ be Fix-to-0. 

By Gadget 2, we realize the signature $[y_b^2+y_c^2,0,y_a^2]$. Similarly by replacing $a$ with $b$ or $c$, we realize $[y_a^2+y_c^2,0,y_b^2]$ and $[y_a^2+y_b^2,0,y_c^2]$ respectively. As $(y_a^2)^4=(y_b^2)^4=(y_c^2)^4$, $|y_b^2+y_c^2|,|y_a^2+y_c^2|,|y_a^2+y_b^2|\in\{0,\sqrt{2}|y_a^2|,2|y_a^2|\}$ and one of them does not equal to 0. Without loss of generality let $|y_b^2+y_c^2|\neq 0$, and we have $[y_b^2+y_c^2,0,y_a^2]\in \mathscr{M}-\mathscr{A}$.
 
\

 \textbf{Case 2: $l=1$.} Let $1\le a,b,c\le k$ be distinct integers and $a$ be the reversed index of $F'$.

  \textbf{Gadget 1:} Construct a generalized mating gadget. Let the $a$th variable be Dangling, the $b$th variable be Sum-up and all other variables of $F'$ be Fix-to-0. 
 
 By Gadget 1, we realize the signature $[y_a^2,0,y_b^2]$, and we are done unless $(y_a^2)^4=(y_b^2)^4$. Similarly, by replacing $b$ with $c$ in the construction of the above gadget, we are done unless $(y_a^2)^4=(y_c^2)^4$. Now we assume $(y_a^2)^4=(y_b^2)^4=(y_c^2)^4$.

\textbf{Gadget 2:} Construct a generalized mating gadget. Let the $a$th variable be Dangling, the $b$th and the $c$th variable be Sum-up and all other variables of $F'$ be Fix-to-0. 

By Gadget 2, we realize the signature $[y_a^2,0,y_b^2+y_c^2]$. Similarly by replacing $a$ with $b$ or $c$, we realize $[y_a^2+y_c^2,0,y_b^2]$ and $[y_a^2+y_b^2,0,y_c^2]$ respectively. Similar to the analysis in Case 1, we are done.
  
\

 \textbf{Case 3: $l=2$.} Let $1\le a,b,c\le k$ be distinct integers and $b,c$ be the reversed indices of $F'$.
 
 We first realize the $[0,0,1]$ signature using $F$ and $[1,0]$. By connecting a $[1,0]$ to each variable of $F$ representing $[1,0]$, a $[1,0]$ to each upright variable of $F'$ and a $[1,0]$ to the $b$th variable, we realize the $[0,1]^{\otimes q+1}$ signature. By making $\lfloor q/2\rfloor$ self-loops on $[0,1]^{\otimes q+1}$, we realize either $[0,1]$ or $[0,0,1]$. If we realize $[0,1]$, we can realize $[0,0,1]=[0,1]^{\otimes 2}$ as well by tensor production.

 \textbf{Gadget 1:} Construct a generalized mating gadget. Let the $a$th variable be Dangling, the $b$th variable be Sum-up, the $c$th variable be Fix-to-1 and all other variables of $F'$ be Fix-to-0. 
 
 By Gadget 1, we realize the signature $[y_b^2,0,y_a^2]$, and we are done unless $(y_a^2)^4=(y_b^2)^4$. Similarly, by replacing $b$ with $c$ in the construction of the above gadget, we are done unless $(y_a^2)^4=(y_c^2)^4$. Now we assume $(y_a^2)^4=(y_b^2)^4=(y_c^2)^4$.

\textbf{Gadget 2:} Construct a generalized mating gadget. Let the $a$th variable be Dangling, the $b$th and the $c$th variable be Sum-up and all other variables of $F'$ be Fix-to-0. 

By Gadget 2, we realize the signature $[y_b^2+y_c^2,0,y_a^2]$. Similarly by replacing $a$ with $b$ or $c$, we realize $[y_b^2,0,y_a^2+y_c^2]$ and $[y_c^2,0,y_a^2+y_b^2]$ respectively. Similar to the analysis in Case 1, we are done.

\

 \textbf{Case 4: $l\ge 3$.}  Let $1\le a,b,c\le k$ be distinct reversed indices of $F'$.

We first realize the $[0,0,1]$ signature using $F$ and $[1,0]$. By connecting a $[1,0]$ to each variable of $F$ representing $[1,0]$, a $[1,0]$ to each upright variable of $F'$ and a $[1,0]$ to the $c$th variable, we realize the $[0,1]^{\otimes q+l-1}$ signature. By making $\lfloor (q+l-2)/2\rfloor$ self-loops on $[0,1]^{\otimes q+l-1}$, we realize either $[0,1]$ or $[0,0,1]$. Again, if we realize $[0,1]$, we can realize $[0,0,1]=[0,1]^{\otimes 2}$ as well by tensor production.

  \textbf{Gadget 1:} Construct a generalized mating gadget. Let the $a$th variable be Dangling, the $b$th variable be Sum-up,  all other reversed variables of $F'$ be Fix-to-1 and all other upright variables of $F'$ be Fix-to-0. 
 
 By Gadget 1, we realize the signature $[y_a^2,0,y_b^2]$, and we are done unless $(y_a^2)^4=(y_b^2)^4$. Similarly, by replacing $b$ with $c$ in the construction of the above gadget, we are done unless $(y_a^2)^4=(y_c^2)^4$. Now we assume $(y_a^2)^4=(y_b^2)^4=(y_c^2)^4$.

\textbf{Gadget 2:} Construct a generalized mating gadget. Let the $a$th variable be Dangling, the $b$th and the $c$th variable be Sum-up,  all other reversed variables of $F'$ be Fix-to-1 and all other upright variables of $F'$ be Fix-to-0. 

By Gadget 2, we realize the signature $[y_a^2,0,y_b^2+y_c^2]$. Similarly by replacing $a$ with $b$ or $c$, we realize $[y_b^2,0,y_a^2+y_c^2]$ and $[y_c^2,0,y_a^2+y_b^2]$ respectively. Similar to the analysis in Case 1, we are done.

\section{Conclusions and future directions}\label{conclusion}
In this article, we prove a dichotomy for Pl-$\#R_D$-CSP, and transform the Pl-\#CSP dichotomies into \#CSP dichotomies  on planar graphs. We present the sufficient and necessary condition for a matchgate signature $f$ and a permutation $\pi$ such that $\pi(f)$ is a matchgate signature as well, which can be checked in polynomial time. We also define the concept of permutable matchgate signatures, and characterize them in detail.

There are several topics we are interested in based on these results. In a related article to appear, we prove dichotomies for \#CSP on minor-closed graph classes, based on dichotomies for \#CSP on planar graphs. In that article, characterization for permutable matchgate signatures also plays an important role in both the algorithm part and the hardness proof.

Furthermore, we are interested in discovering an alternative algebraic algorithm for the FKT algorithm. In particular, we consider the following problem, which has the potential to reduce the algorithm's dependence on a certain plane embedding. Given a matchgate signature $f$, can we decide whether there exist a permutation $\pi$ such that $\pi(f)$ is also a matchgate signature in polynomial time? If so, can we find such permutation in polynomial time? 





\bibliography{ref}

\appendix
\section{A dichotomy for Pl-\#$R_D$-CSP}\label{sec:plrdcsp}

In this section, we prove Theorem \ref{plRDCSP}, the dichotomy for Pl-$\#R_D$-CSP. It is stated as follows.

\begin{theorem}
    Suppose $D\ge3$ is a constant. If $\mathcal{F}\subseteq \mathscr{A}$ or $\mathcal{F}\subseteq \mathscr{P}$ or $\mathcal{F}\subseteq \widehat{\mathscr{M}}$, Pl-$\#R_D$-CSP$(\mathcal{F})$ is computable in polynomial time; otherwise it is $\text{\#P}$-hard.
\end{theorem}

The proof of Theorem \ref{plRDCSP} is highly related to the concept of non-degenerate signatures. A signature $f$ of arity $k$ is said to be a \textit{degenerate signature} if it has the form $f=[a_1,b_1]\otimes[a_2,b_2]\otimes\dots\otimes[a_k,b_k]$; otherwise it is said to be a \textit{non-degenerate signature}. In particular, a binary signature is non-degenerate if and only if it is of rank 2.

In this proof, we again introduce the concept of a generalized mating gadget. We revise this concept in the following paragraph.

"In a \textit{generalized mating gadget}, the variables of $F$ are divided into four parts: \textit{Sum-up} variables, \textit{Fix-to-0} variables, \textit{Fix-to-1} variables and a single \textit{Dangling} variable. We assign $F$ to each vertex of degree $n$ labelled by a solid circle in Figure \ref{fig:gadhandshake2} in the following way: for each $1\le a\le n$, if the $a$th variable is the Dangling variable, we let it correspond to the dangling edge. Otherwise we connect it to a vertex $v_a$ of degree 2 labelled by a hollow square. Furthermore, if the $a$th variable is a Sum-up/Fix-to-0/Fix-to-1 variable, we assign a $[1,0,1]$/$[1,0,0]$/$[0,0,1]$ signature to $v_a$, and the gadget become well defined.
 We remark that using $[1,0]$ we may obtain the $[1,0,0]=[1,0]^{\otimes 2}$ signature, and we construct a gadget of $[0,0,1]$ when needed. "
 
\begin{figure*}
	\centering	\includegraphics[width=0.45\textwidth]{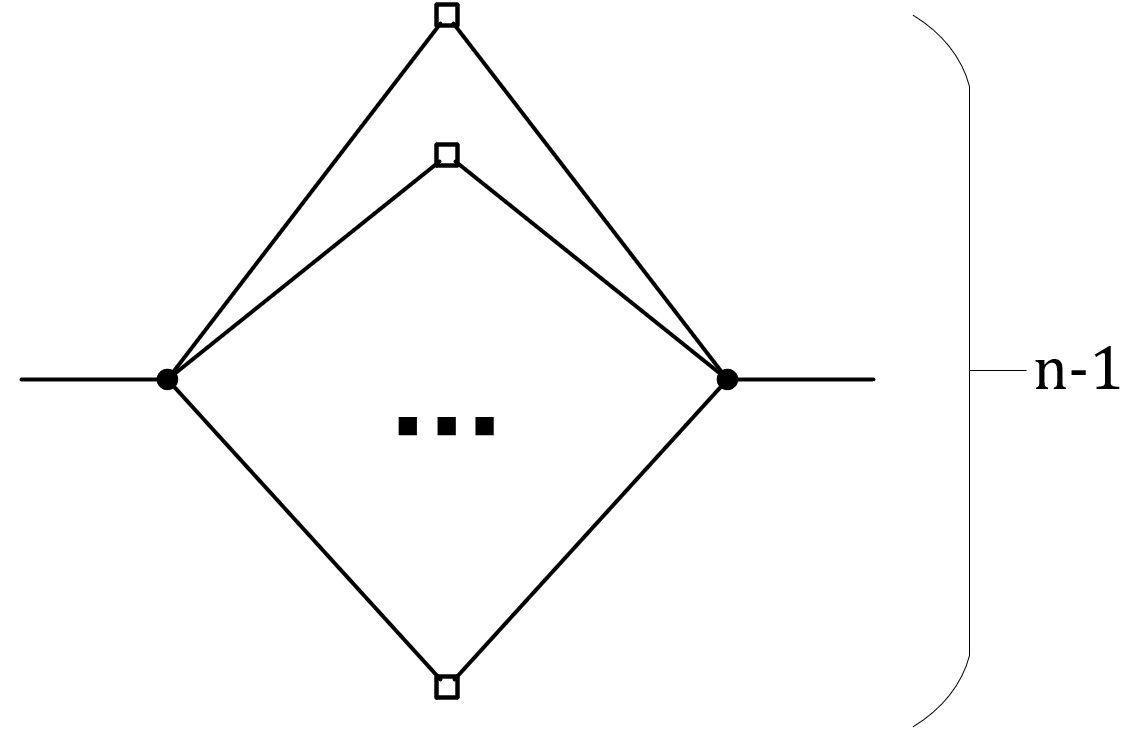}
	\caption{The construction of the generalized mating gadget.}
    \label{fig:gadhandshake2}
\end{figure*} 
The proof of Theorem \ref{plRDCSP} involves three lemmas, similar to that of Theorem \ref{genRDCSP} originally presented in in \cite[Section 6]{cai2014complexity}. In fact, the proof of \cite[Lemma 6.1]{cai2014complexity} can be straight forwardly extends to the first lemma involved, stated as follows.
\begin{lemma}
    Suppose $g$ is a non-degenerate binary signature. Then Pl-\#$R_3\text{-CSP}(\mathcal{F}\cup\{g,=_2\})\le_T$Pl-\#$R_3\text{-CSP}(\mathcal{F}\cup\{g\})$.
    \label{lem:plr3obtain=2}
\end{lemma}

The second lemma corresponds to \cite[Lemma 6.5]{cai2014complexity}. 
\begin{lemma}
    Suppose $f\in \mathcal{F}$ of arity $k$ is a non-degenerate signature. Then one of the following holds:
    \begin{itemize}
        \item There exists a non-degenerate binary signature $g$ such that Pl-Holant$(\mathcal{F}\cup\{g\}|[1,0],[1,0,1])\le_T$Pl-Holant$(\mathcal{F}|[1,0],[1,0,1])$.
        \item Pl-Holant$(\mathcal{F}\cup\{[0,0,1]\}|\{[1,0],[1,0,1]\})\le_T$Pl-Holant$(\mathcal{F}|\{[1,0],[1,0,1]\})$.
    \end{itemize}
    \label{lem:plr3obtaind1}
\end{lemma}

The proof of Lemma \ref{lem:plr3obtaind1} is also similar to that of \cite[Lemma 6.5]{cai2014complexity}, with the exception that we realize $[0,0,1]=[0,1]^{\otimes 2}$ instead of $[0,1]^{\otimes m}$.
Once we realize $[0,1]^{\otimes m}$, we use 2 copies of $[0,1]^{\otimes m}$ to construct a generalized mating gadget. We let one variable be Dangling and all other $m-1$ variables be Sum-up. By this gadget, we realize $[0,0,1]$ in a planar way. It can be verified that with this change, all gadgets constructed become planar and the proof still holds in the planar setting. For clarity, we provide a complete proof, though most analysis are similar to that of \cite[Lemma 6.5]{cai2014complexity}.

\begin{proof}
    We prove this by induction. If $k=2$, the first statement trivially holds. Now suppose this lemma holds for all $k<n$ and consider the case when $k=n$. If there exists $1\le i\le k$ such that $f^{x_i=0}$ is non-degenerate, we are done by induction. If for all $1\le i\le k$, $f^{x_i=0}$ is identically 0, then $f=c[0,1]^{\otimes k}$, which is a contradiction since $f$ is non-degenerate. Consequently, we may suppose there exists $1\le i\le k$ such that $f^{x_i=0}$ is degenerate but not identically 0.

    Without loss of generality, we let $i=1$ and we have $f^{x_1=0}=[a_2,b_2]\otimes[a_3,b_3]\otimes\dots\otimes[a_k,b_k]$ and $[a_j,b_j]\neq [0,0]$ for each $2\le j\le k$. Let $L=\{l|2\le l\le k, a_l\neq 0\}$. If $|L|\neq k-1$,  then by further fixing $x_l$ to 0 through $[1,0]$ we realize the signature $h=[0,1]^{\otimes k-1-|L|}$. We use 2 copies of $h$ to construct a generalized mating gadget. We let one variable be Dangling and all other $k-2-|L|$ variables be Sum-up. By this gadget, we realize $[0,0,1]$ in a planar way and the second statement holds.

    For convenience, we always ignore a constant factor in the following analysis. Consequently we may assume $f^{x_1=0}=[1,b_2]\otimes[1,b_3]\otimes\dots\otimes[1,b_k]$ and $f^{x_2=0}=c[1,b_1']\otimes[1,b_3']\otimes\dots\otimes[1,b_k']$. As $f^{x_1=0}(0\dots0)=f^{x_2=0}(0\dots0)=f(0\dots0)$, $c=1$. For an arbitrary string $\beta=\beta_2\dots \beta_k\in \{0,1\}^{k-1}-\{1\dots1\}$, without loss of generality we assume $\beta_2=0$. Then we have $f^{x_1=1}(\beta)=f^{x_2=0}(1\beta_3\dots \beta_k)=b_1'f^{x_2=0}(0\beta_3\dots \beta_k)=b_1'f^{x_1=0}(\beta)$. Let $b_1=b_1'$, and consequently we have
    $$f=[1,b_1]\otimes f^{x_1=0}+d[0,1]^{\otimes k}=[1,b_1]\otimes\dots\otimes[1,b_k]+d[0,1]^{\otimes k}$$
    And as $f$ is non-degenerate, $d\neq 0$.

    Using $k-1$ [1,0] we may realize $p_i=f^{x_1=\dots=x_{i-1}=x_{i+1}=\dots=x_k=0}=[1,b_i]$. 
    If there exists $i$ such that $b_i=\pm \mathfrak{i}$, by connecting $p_i$ to the $i$th variable of $f$ via [1,0,1], we realize $h=[0,1]^{\otimes k-1}$ up to a constant factor $cb_i$. Again, by using 2 copies of $h$ to construct a generalized mating gadget we realize $[0,0,1]$, and the second statement holds.

    Otherwise, for $3\le i\le k$, by connecting $p_i$ to the $i$th variable of $f$ via [1,0,1], we realize $g=\begin{pmatrix}
        1 & c_1\\
        c_2 & c_1c_2+d'
    \end{pmatrix}$ up to a constant factor for some $c_1,c_2,d'$ satisfying $d'\neq 0$. As $g$ is of rank 2, $g$ is non-degenerate, and consequently the first statement holds.
\end{proof}
The third lemma extends the result in \cite[Lemma 6.2]{cai2014complexity}. Instead of realizing a single non-degenerate $h$, we use a generalized mating gadget to realize $h^Th$ in a planar way. As $h$ is non-degenerate, $h^Th$ is non-degenerate as well and this completes our proof.
\begin{lemma}
     Suppose $f\in \mathcal{F}$ of arity $k$ is a non-degenerate signature. Then there exists a non-degenerate binary signature $g$ such that Pl-Holant$(\mathcal{F}\cup\{g\}|[1,0],[0,0,1],[1,0,1])\le_T$Pl-Holant$(\mathcal{F}|[1,0],[0,0,1],[1,0,1])$.
    \label{lem:plr3obtainh}
\end{lemma}
\begin{proof}
    We only need to prove that for each non-degenerate $f$, we can use 2 copies of $f$ to construct a generalized mating gadget which realizes a non-degenerate binary signature $g$ in a planar way.     
    We prove this by induction. If $k=2$, we let one variable be Dangling and the other variable be Sum-up. By this gadget, we realize a non-degenerate $g$ as $f$ is non-degenerate.
    
    Now suppose this lemma holds for all $k<n$ and consider the case when $k=n$. If there exists $1\le i\le k$ such that $f^{x_i=0}$ is non-degenerate, by induction we may use 2 copies of $f^{x_i=0}$ to construct a generalized mating gadget which realizes a non-degenerate binary signature $g$. Then by letting the $i$th variable of $f$ be Fix-to-0, we can use 2 copies of $f$ to construct a generalized mating gadget to realize $g$ in the similar way and we are done. Similarly if there exists $1\le i\le k$ such that $f^{x_i=1}$ is non-degenerate, by letting the $i$th variable of $f$ be Fix-to-1 we can realize a non-degenerate $g$ as well.

    Now we may assume $f^{x_i=c}$ is degenerate for all $1\le i\le k, c\in\{0,1\}$. If $f^{x_1=0}$ remains constant at 0, then $f=[0,1]\otimes f^{x_1=1}$. As $f^{x_1=1}$ is degenerate, $f$ is also degenerate, which is a contradiction. Consequently we may suppose there exists a string $\alpha=\alpha_2\dots\alpha_k\in \{0,1\}^{k-1}$ such that $f^{x_1=0}(\alpha)\neq 0$. 
    
    For each $\beta=\beta_2\dots \beta_k\in \{0,1\}^{k-1}$ satisfying $\beta\neq \overline{\alpha}$, without loss of generality we assume $\alpha_2=\beta_2=c\in \{0,1\}$. Then $f^{x_2=c}=[1,\lambda]\otimes[a_3,b_3]\otimes\dots\otimes[a_k,b_k]$ as $f^{x_2=c}(0\alpha_3\dots\alpha_k)\neq 0$. Here, $\lambda=\frac{f(1\alpha)}{f(0\alpha)}$. Now we have 
    $$f^{x_1=1}(\beta)=f^{x_2=c}(1\beta_3\dots \beta_k)=\lambda f^{x_2=c}(0\beta_3\dots \beta_k)=\lambda f^{x_1=0}(\beta)$$
    
    If there exists a $\gamma\in \{0,1\}^{k-1}$ satisfying $\gamma\neq \alpha,\overline{\alpha}$ and $f^{x_1=0}(\gamma)\neq 0$, then by the similar analysis we have 
    $$f^{x_1=1}(\overline{\alpha})=\frac{f(1\gamma)}{f(0\gamma)}f^{x_1=0}(\overline{\alpha})=\lambda f^{x_1=0}(\overline{\alpha})$$
    This indicates that $f^{x_1=1}(\beta)=\lambda f^{x_1=0}(\beta)$ for each $\beta\in \{0,1\}^{k-1}$, and consequently $f=[1,\lambda]\otimes f^{x_1=0}$ is degenerate, which is a contradiction.

    As a result, for each $\gamma\in \{0,1\}^{k-1}$ satisfying $\gamma\neq \alpha,\overline{\alpha}$, $f^{x_1=0}(\gamma)=0$. Besides, $f^{x_1=1}(\gamma)=\lambda f^{x_1=0}(\gamma)=0$ as well. Consequently, if $\delta\neq 0\alpha,1\alpha,0\overline{\alpha},1\overline{\alpha}$, then $f(\delta)=0$. Since $f^{x_1=0}$ is degenerate, we may write $f^{x_1=0}=[a_2,b_2]\otimes\dots\otimes[a_k,b_k]$. If $f(0\alpha) f(0\overline{\alpha})\neq 0$, then $\prod_{i=1}^k a_ib_i\neq 0$. Since $k\ge 3$, there exists a $\delta \neq \alpha,\overline{\alpha}$ satisfying $f(0\delta)\neq 0$, a contradiction. Thus $f(0\alpha) f(0\overline{\alpha})=0$. Similarly $f(1\alpha) f(1\overline{\alpha})=0$.

    Since $f(0\alpha) \neq 0$, $f(0\overline{\alpha})=0$. If $f(1\overline{\alpha})=0$, then $f=[f(0\alpha),f(1\alpha)]\otimes f^{x_1=0}$ is degenerate, a contradiction. Consequently $f(1\overline{\alpha})\neq 0$, $f(1\alpha)=0$. We use 2 copies of $f$ to construct a generalized mating gadget. We let $x_1$ be Dangling and all other $k-1$ variables be Sum-up. By this gadget, we realize $g=[f^2(0\alpha),0,f^2(1\overline{\alpha})]$ which is a non-degenerate signature.
\end{proof}

Now we are ready to prove Theorem \ref{plRDCSP}.

\begin{proof}[Proof of Theorem \ref{plRDCSP}]
    When $\mathcal{F}\subseteq \mathscr{A}$ or $\mathcal{F}\subseteq \mathscr{P}$ or $\mathcal{F}\subseteq \widehat{\mathscr{M}}$, the polynomial-time algorithms of these tractable cases can still be applied, hence Pl-$\#R_D$-CSP$(\mathcal{F})$ is still computable in polynomial time in these cases. Now we suppose none of these statements holds for $\mathcal{F}$. As a result, there exists a non-degenerate signature $f\in \mathcal{F}$.
    
    By connecting a $=_k$ signature to a $=_3$ signature via $=_2$, we realize $=_{k+1}$. By induction, we can realize $=_k$ for arbitrary $k$ given $=_2$ and $=_3$. Consequently, the following reduction holds.
    $$\text{Pl-\#CSP}(\mathcal{F})\le_T \text{Pl-}R_D\text{-\#CSP}(\mathcal{F}\cup\{=_2\})$$

    By Lemma \ref{lemcsp=hol} and Theorem \ref{thmHT}, $\text{Pl-}R_D\text{-\#CSP}(\mathcal{F}\cup\{=_2\})\equiv_T \text{Pl-Holant}(\mathcal{F}\cup\{=_2\}|\{=_1,\dots,=_D\})\equiv_T \text{Pl-Holant}(\widehat{\mathcal{F}}\cup\{=_2\}|\{[1,0],[1,0,1],\dots,\widehat{=_D}\})$. As $f\in \mathcal{F}$ is non-degenerate, $\widehat{f}$ is non-degenerate as well, since otherwise a contradiction would occur. Consequently, by Lemma \ref{lem:plr3obtain=2}, Lemma \ref{lem:plr3obtaind1} and Lemma \ref{lem:plr3obtainh}, there exists a non-degenerate binary signature $h$ such that $\text{Pl-Holant}(\widehat{\mathcal{F}}\cup\{=_2\}|\{[1,0],[1,0,1],\dots,\widehat{=_D}\})\le_T \text{Pl-Holant}(\widehat{\mathcal{F}}\cup\{h\}|\{[1,0],[1,0,1],\dots,\widehat{=_D}\})\le_T \text{Pl-Holant}(\widehat{\mathcal{F}}|\{[1,0],[1,0,1],\dots,\widehat{=_D}\})$. Again, by Lemma \ref{lemcsp=hol} and Theorem \ref{thmHT} we have $$\text{Pl-Holant}(\widehat{\mathcal{F}}|\{[1,0],[1,0,1],\dots,\widehat{=_D}\})\equiv_T \text{Pl-Holant}(\mathcal{F}|\{=_1,\dots,=_D\})\equiv_T \text{Pl-}R_D\text{-\#CSP}(\mathcal{F}).$$

    The above analysis shows that $\text{Pl-\#CSP}(\mathcal{F})\le_T\text{Pl-}R_D\text{-\#CSP}(\mathcal{F})$. When none of $\mathcal{F}\subseteq \mathscr{A}$, $\mathcal{F}\subseteq \mathscr{P}$ and $\mathcal{F}\subseteq \widehat{\mathscr{M}}$ holds, $\text{Pl-\#CSP}(\mathcal{F})$ is \#P-hard, as well as $\text{Pl-}R_D\text{-\#CSP}(\mathcal{F})$. This completes the proof.
\end{proof}

\end{document}